\documentclass[reqno,11pt]{amsart}
\usepackage{amsmath,amssymb,amsthm,tikz,float,enumitem,mathtools}
\numberwithin{equation}{section}
\usepackage[margin=2.5cm]{geometry}
\usepackage[bottom]{footmisc}
\usepackage[linkcolor=blue,colorlinks=true]{hyperref}
\usepackage{hyperref}
\usepackage[capitalise]{cleveref}

\allowdisplaybreaks

\def\ds{\displaystyle}
\def\p{\partial}
\def\t{\tilde}

\subjclass[2010]{ 93B05, 93B07, 93B60, 35K05, 35P15 }
\keywords{parabolic systems, control, observability, homogenization, Carleman estimates}
\makeatletter
\def\l@subsection{\@tocline{2}{0pt}{2pc}{5pc}{}}
\makeatother

\newcommand\numberthis{\addtocounter{equation}{1}\tag{\theequation}}

\newcommand{\mc}[1]{\mathcal{#1}}
\newcommand{\mb}[1]{\mathbb{#1}}
\newcommand{\mf}[1]{\mathfrak{#1}}

\newcommand{\rd}{{\rm d}}

\theoremstyle{definition}

	\newtheorem{theorem}{Theorem}
	
	\newtheorem{lemma}[theorem]{Lemma}
	\newtheorem{prop}[theorem]{Proposition}
	\newtheorem{defn}[theorem]{Definition}
	
	\newtheorem{rem}[theorem]{Remark}
	
	\numberwithin{theorem}{section}

\title[Control and homogenization]{Control and homogenization of a system of coupled parabolic equations with an oscillating coefficient}

\author{Vaibhav Kumar Jena and Abu Sufian}
\address{Vaibhav Kumar Jena
	\newline \indent
	{TIFR Centre for Applicable Mathematics, \newline \indent
		560065 Bangalore, Karnataka, India.}}
\email{vkjena22@tifrbng.res.in}

\address{Abu Sufian
	\newline \indent
	{Fraunhofer ITWM, \newline \indent
		67663 Kaiserslautern, Germany}}
\email{abu.sufian@itwm.fraunhofer.de}

\begin{document}

\begin{abstract}
In this article, we study the uniform null controllability problem for a system of coupled parabolic equations with an oscillating coefficient. This is done in three steps---first, we study the spectral properties of an elliptic operator; second, we allow the system to evolve freely and obtain the required decay; third, we use a Carleman estimate to prove a suitable observability result. This uniform null controllability property is then used to homogenize the associated  coupled parabolic system.
\end{abstract}

\maketitle


\section{Introduction}

The study of control and homogenization of systems with oscillating coefficients has gained significant attention in the realm of applied mathematics and engineering. In many physical phenomena, such as heat transfer in materials or diffusion processes, the coefficients governing the dynamics of the system may exhibit oscillatory behaviour. Understanding the behaviour and developing effective control strategies for such systems is crucial in such fields. Moreover, while dealing with the controllability problem in a heterogeneous medium, it is often more convenient to transform the problem into a control problem given on a homogeneous medium using the mathematical theory of homogenization.

Let us discuss some background for the problem considered in this article. Homogenization of control problems has been a very active topic in the research community for the past few decades. To the best of the authors' knowledge, one of the earliest works in this direction is in \cite{zuazua94}, where homogenization of an approximate controllability problem is studied for a linear parabolic equation with oscillating diffusion coefficient matrix. Later, this work was extended in a perforated domain in \cite{don-nabil01}.  
Following this, there have been several works in homogenization of approximate controllability for evolutionary equations with oscillating coefficients in rough interface domains, two component domains, and perforated domains. For results related to such problems, we refer to \cite{jose15, jose21, Fella19, ped06} and the references therein.

However, there are only a few results that deal with homogenization of null controllability problems for parabolic equations with oscillating periodic coefficients. In this direction, we mention the earliest finding in \cite{lopez_zuazua}, where  homogenization of null controllability is analysed for a parabolic equation in one dimension with a periodic oscillating coefficient. Furthermore, in \cite{teb12}, the authors address homogenization of a null control problem for parabolic equations with oscillating coefficient in $n$-dimension, by applying vanishing null control in the full domain at the microlevel.  

The present work is devoted to take the work in \cite{lopez_zuazua} one step further, by extending the control and homogenization study to a \emph{system} of coupled parabolic equations.

\subsection{Problem}
Let $m,n \in \mathbb{N}$, and let $ A \in M_{n \times n} (\mathbb{R}) $ and $B \in M_{n \times m} (\mathbb{R})$ be given matrices. 
We consider a function $a\in W^{2,\infty}(\mathbb{R})$, that is $1$-periodic, with the following uniform upper and lower bound
\begin{equation} \label{upper_lower_bound_1}
0<a_m\leqslant a(x)\leqslant a_M, \quad \forall x\in \mathbb{R}.
\end{equation}
For $\varepsilon \in (0,1)$, let us define the $\varepsilon$-periodic function $a^\varepsilon$ as $a^\varepsilon(x)=a\left(\frac{x}{\varepsilon}\right)$. For convenience, we will use the notation $\mc{L}^\varepsilon$ to denote the following operator
\begin{equation} \label{eq_L-op}
\mc{L}^\varepsilon := \p_x (a^\varepsilon(x)\p_x) I_{n\times n},
\end{equation}
where $I_{n\times n}$ denotes the $n\times n$ identity matrix.
We will write $\tilde{u}=(\tilde{u}_1,\tilde{u}_2,\cdots,\tilde{u}_n)^\text{tr} \in \mathbb{R}^n$, where tr denotes the transpose of the vector. Let $T>0$, and let $\omega \subset (0,1)$, be a non-empty open subset. For any function $f \in (L^2((0,T)\times \omega))^m,$ consider the following evolutionary system
\begin{align}\label{general_B_M}
\begin{dcases}
\p_t \tilde{u}^\varepsilon-\mc{L}^\varepsilon \tilde{u}^\varepsilon + A \tilde{u}^\varepsilon= B f 1_{\omega}, & (t,x)\in (0, T)\times (0,1), \\ 
\tilde{u}^\varepsilon (t,0)=\tilde{u}^\varepsilon(t,1)=0, & t \in (0,T),\\
\tilde{u}^\varepsilon (0,x)=\tilde{u}^0, & x \in (0,1),
\end{dcases}
\end{align}
where $1_{\omega}$ represents the characteristic function over $\omega$. Here, $A$ is the coupling matrix and $B$ is the control matrix.
We are interested in studying the controllability problem for the above system, hence, we now provide the definition of controllability.
\begin{defn}
Given any $T>0$ and $\tilde{u}_0\in (L^2(0,1))^n,$ the system \eqref{general_B_M} said to be null controllable at time $T,$ if there exists a control $f\in L^2(\omega)$, such that the solution $\tilde{u}^\varepsilon$ satisfies
\begin{align}\label{null}
\tilde{u}^\varepsilon(T,x)=0, \quad \text{ for } x\in (0,1).
\end{align} 
In this case, the function $f$ is called a null control for system \eqref{general_B_M}.
\end{defn}
Note that, in system \eqref{general_B_M}, since $B \in M_{n \times m} (\mathbb{R})$, we are using $m$-controls to control the system. In applications, one usually assumes that $m\leqslant n$. Furthermore, the presence of the term $1_{\omega}$ in \eqref{general_B_M} implies that the control $f$ only acts in the region $\omega$.
Before presenting the main results we mention the following hypotheses on the coupling matrix $A$ and the control matrix $B$, that will be assumed throughout the article:
\begin{itemize}

\item[(H1)] The matrix pair $(A,B)$ satisfies the Kalman condition, i.e., the Kalman matrix $[A|B] = (B|AB|A^2B|\ldots|A^{n-1}B)$ has rank $n$. 

\item[(H2)] The matrix $A$ has $n$-distinct eigenvalues and $\ds \min \{\text{Spec}(A)\} > -\frac{\pi^2}{a_M}$. 
\end{itemize}

The need of the above assumptions is explained later in the article. Then, our main results are as follows.

\begin{theorem}[Controllability] \label{thm_control_m}
Assume that (H1) and (H2) are satisfied.
Let $T>0$ and let $u_0 \in (L^2(0,1))^n$ be fixed. Then for every $\varepsilon\in (0,1)$, there exists a null control function $f^\varepsilon\in (L^2((0,T)\times\omega))^m$  to the system \eqref{general_B_M}, that is, the solution $\tilde{u}^\varepsilon$ satisfies \eqref{null}.
Moreover, there exists a constant $C := C(T) >0$, independent of $\varepsilon$, such that
\[\| f^\varepsilon \|_{L^2((0,T)\times \omega)^m} \leqslant C(T) \| u^0 \|_{(L^2(0,1))^m}.\]
\end{theorem}

\begin{theorem}[Homogenization] \label{thm_homogen_m}
Assume that (H1) and (H2) are satisfied. Let $a \in W^{2,\infty}(\mathbb{R})$ be $1$-periodic and further assume that it satisfies the assumption \eqref{upper_lower_bound_1}. Then, there exists a sequence of null control functions $\{f_\varepsilon\}_{\varepsilon\in (0,1)}$ to the system \eqref{general_B_M}, such that 
$$f_\varepsilon\to f_0 \text{ strongly in } L^2((0,T)\times\omega )^m, $$
where $f_0$ is a  null control to the the homogenized system of \eqref{general_B_M}, that is
\begin{align}\label{eq_homogen}
\begin{dcases}
\p_t \tilde{u}-\mc{L}^0 \tilde{u} +A \tilde{u}= B f_0 1_{\omega}, & (t,x)\in (0, T)\times (0,1), \\ 
\tilde{u}(t,0)=\tilde{u}(t,1)=0, & t \in (0,T),\\
\tilde{u} (0,x)=\tilde{u}^0, \tilde{u}(T,x)=0, & x \in (0,1),
\end{dcases}
\end{align}
where the homogenized operator $\mathcal{L}^0 := \left(\left(\ds \int_0^1\frac{1}{a(s)}ds\right)^{-1}\partial_{xx} \right) I_{n\times n}.$
\end{theorem}

\begin{rem}
The assumption (H2) on $A$ can be relaxed further. One can get the same result under the assumption that the eigenvectors of $A$ form a basis for $\mathbb{R}^n$, with slight modifications to the current proof. Also, note that the condition on Spec($A$) is a weaker requirement than the positive definiteness of $A$. In particular, our result holds for a class of matrices larger than the set of positive definite matrices.
\end{rem}

\subsection{Proof strategy} \label{ssec_proof_strat}
To establish uniform null controllability, the main idea is to follow the strategy outlined in \cite{lopez_zuazua}. In particular, we partition the time interval $(0,T)$, into three equal parts: $[0, T/3]$, $[T/3, 2T/3]$, and $[2T/3, T]$, and then address each interval separately.

In the first time interval $[0, T/3]$, we investigate the spectral properties for the associated  adjoint operator. Employing spectral analysis, we define an $\varepsilon$-dependent low-frequency space and establish the uniform null controllability of the projection over the low-frequency space. This is achieved by demonstrating an $\varepsilon$-independent observability inequality. Here, assumption (H2) is essential to show that the eigenvalues for the required operator are positive and that we can get a basis of eigenvectors for the considered vector space.

In the second time interval $[T/3, 2T/3]$, we begin with initial data at $T/3$ using the solution from the first time interval. We allow the system to evolve freely without applying any control. Utilizing the spectral properties studied in the first part, we show that the solution exhibits decay of order $e^{-c_1/\varepsilon^2}$, for some constant $c_1>0$.

During the third time interval $[2T/3, T]$, we prove an observability estimate for the solution using a parabolic Carleman estimate. The observability constant is of the order $e^{c_2/ \varepsilon^{4/3}}$. However, for this part, we consider initial data at $2T/3$ using the solution from the second part, which is of the order $e^{-c_1/\varepsilon^2}$. Consequently, the null control cost in the third time interval is uniformly bounded. Using the form bound, the considered coupled system will be homogenized.

In this study, we have made two significant advancements. Firstly, we have successfully established the necessary spectral behavior property for the coupled elliptic eigenvalue problem, which is crucial for solving the homogenization of the controllability problem of the coupled parabolic system. This enables us to prove the uniform low frequency null controllability and suitable decay estimates. Secondly, using a parabolic Carleman estimate (see \cite{coron_book,fur_iman,tucsnak_weiss}) we have proven a Carleman-type weighted estimate for the coupled parabolic system (see \eqref{eq_CE_nxn_A6} or \eqref{eq_CE_nxn_B6}), which is then used to prove the required observability result. 

\subsection{Outline} We now provide an outline of the paper. In \cref{sec_mto1}, we show that studying the control problem for \eqref{general_B_M} is equivalent to studying a control problem for a system with only one control acting on it and the coupling given by a cascade matrix. This later problem is simpler to study. In \cref{sec_mainproof}, we prove observability for the considered system using the three part proof idea described above. In the appendix we give details of some technical results used in the proofs. In paticular, in \cref{app_eigen}, we recall the spectral behaviour  for the $1$-dimensional elliptic eigenvalue problem with oscillating coefficient. In \cref{app_mcontrol}, we show the reduction of the $m$-control system to a system with a finite number of cascade subsystems each controlled by one control.
 
\subsection{Acknowledgement}
The authors would like to thank Debayan Maity, for his advice and helpful suggestions concerning this work. The authors are also grateful to TIFR Centre for Applicable Mathematics for providing financial support. A part of this work was completed while the second author was employed at TIFR Centre for Applicable Mathematics.

\section{$m$-controls to $1$-control} \label{sec_mto1}

In our proof, we will be transforming system \eqref{general_B_M} into a cascade form. In particular, the coupling matrix $A$ will be transformed into a cascade matrix $\mc{C}$. The transformed problem is comparatively easier to study. Although we use the controllability result to obtain homogenization, the topic of control of cascade parabolic systems is interesting in itself, and there are several results in the literature along this direction. In \cite{MR2598471}, the authors study a null controllability problem for a general cascade parabolic system. In \cite{MR2141924}, an approximate controllability problem for a parabolic system with a diagonalisable diffusion matrix is addressed. The work in \cite{MR2511553} presents a generalised Kalman condition for the $n\times n$ linear parabolic system with constant coefficients and diagonalisable diffusion matrix. Next, \cite{khodjaBDG} presents a control result for coupled parabolic equations with time dependent coupling matrices. For further reading related to the null controllability of coupled systems, we refer \cite{coron09,coron10} and the references therein. We also refer to the notable article \cite{MR2846087} for an excellent survey on the controllability of systems of parabolic equations. Finally, we mention the recent work \cite{MR4687430}, where the authors present a Kalman condition for the controllability of a coupled Stokes system, using related ideas from \cite{MR2511553}.

In our work the idea to transform the $m$-control system into a combination of (several) smaller one control systems is inspired from the work in \cite{khodjaBDG}. Using this work, we can show that, under the Kalman rank condition, for studying the control problem for system \eqref{general_B_M}, where $m$-controls are used, it is enough to study a related control problem where only one control is used. Indeed, when (H1) holds, system \eqref{general_B_M} can be decomposed into a finite number of \emph{smaller} subsystems or blocks, where the subsystems are controlled by one control each. This process is described in \cref{app_mcontrol}. In particular, we will consider the system
\begin{align}\label{canonical_system_n}
\begin{dcases}
\p_t{u}^\varepsilon- \mc{L}^\varepsilon {u}^\varepsilon +\mathcal{C} u^\varepsilon = D 1_{\omega} f, & (t,x) \in (0, T)\times (0,1), \\ 
u^\varepsilon(t,0)=u^\varepsilon(t,1)=0, & t \in (0,T), \\ 
u(0,x)={u}^0, & x \in (0,1), 
\end{dcases}
\end{align}
where $\mc{C}$ is the coupling matrix and the control matrix $D$ satisfies $D:=(e_{S_1}|e_{S_2}|\ldots|e_{S_r})$, for some indices $S_i$'s, $1\leqslant i \leqslant r$. Here, $e_j$ is a vector from the standard basis of $\mb{R}^n$, with 1 in the $j$-th position and 0 everywhere else. The matrix $\mc{C}$ is composed of $r$-cascade matrices that allow us to simplify the problem. Hence it is enough to study the control problem for a system where only one control is acting.

For the reader's convenience, we present the method to reduce a general one control system to a cascade one control system, in the following section. The case with $m$-controls is an extension of this method and can be found in \cref{app_mcontrol}.

\subsection{One-control case}
Let $b \in \mb{R}^n$ and $ A \in \mc{L}(\mathbb{R}^n)$ such that the pair $(A,b)$ satisfies the Kalman rank condition, that is, the Kalman matrix $P:=[b|Ab|A^2b |\ldots |A^{n-1}b]$ has rank $n$. This condition is the analogous version of (H1). Note that, $P$ is invertible.
Then, we consider the system
\begin{align}\label{general_C_M}
\begin{dcases}
\p_t \tilde{u}^\varepsilon-\mc{L}^\varepsilon \tilde{u}^\varepsilon +A \tilde{u}^\varepsilon=b 1_{\omega} f, & (t,x)\in (0, T)\times (0,1), \\ 
\tilde{u}^\varepsilon (t,0)=\tilde{u}^\varepsilon(t,1)=0, & t \in (0,T),\\
\tilde{u} (0,x)=\tilde{u}^0, & x \in (0,1).
\end{dcases}
\end{align}
Let us use the following change of variable: ${u}=P^{-1}\tilde{u}$.
First, notice that $P^{-1}b= e_1$. Let us choose the matrix given by $\mc{C}=[e_2~e_3~\cdots~e_n~C_n]$, where $\{e_i\}_{i=1}^n$ denotes the standard basis of $\mathbb{R}^n$ and $C_n=(a_1,a_2,\cdots, a_{n})^\text{tr}\in \mathbb{R}^n$ is such that 
$$A^n=a_1 I+ a_2A+\cdots+a_{n}A^{n-1},$$ 
that is, $a_i$'s are the coefficients of the characteristic polynomial of $A$. Hence, $\mc{C}$ can be written as
\[ \mc{C} = \begin{bmatrix}
0 & 0 & \ldots & 0 & a_1 \\
1 & 0 & \ldots & 0 & a_2 \\
0 & 1 & \ldots & 0 & a_3 \\
\vdots & \vdots & \ddots & \ddots & \vdots \\
0 & 0 & \ldots & 1 & a_n
\end{bmatrix}. \]
Then, we have $AP = ( Ab|A^2b|\ldots|A^nb ) = P \mc{C}$, which implies that $P^{-1}AP=\mc{C} $. With the above transformation, system \eqref{general_C_M} takes the following canonical form
\begin{align}\label{canonical_system}
\begin{dcases}
\p_t u^\varepsilon - \mc{L}^\varepsilon u^\varepsilon +\mathcal{C} u^\varepsilon = e_1 1_{\omega} f, & (t,x)\in (0, T) \times (0,1), \\ 
u^\varepsilon(t,0)=u^\varepsilon(t,1)=0, & t \in (0,T),\\
u(0,x)=u^0= P^{-1} \tilde{u}^0, & x \in (0,1).
\end{dcases}
\end{align}
Then we have the following result.
\begin{prop}
The controllabilty of \eqref{general_C_M} is equivalent to the controllabilty of \eqref{canonical_system}.
\end{prop}
\begin{proof}
Since the transformation matrix $P$ is invertible the two notions are equivalent.
\end{proof}
 
As any pair $(A, b)$ that satisfies the Kalman rank condition can be transformed into a canonical system of type \eqref{canonical_system}, our analysis will focus on this canonical form to investigate the null controllability problem. We now provide equivalent statements for the one control system \eqref{general_C_M}, as done for system \eqref{general_B_M}. In particular, corresponding to \cref{thm_control_m} and \cref{thm_homogen_m}, we have the following statements.
\begin{theorem}[Controllability] \label{thm_control_1}
Assume that rank$[b|Ab|A^2b |\ldots |A^{n-1}b] = n$ and that $A$ satisfies (H2). Let $T>0$ and let $u_0\in (L^2(0,1))^n$ be fixed. Then, for every $\varepsilon\in (0,1)$, there exists a null control $f_\varepsilon \in L^2((0,T)\times\omega) $   to the system \eqref{general_C_M}.
Moreover, there exists a constant $C := C(T) >0$, independent of $\varepsilon$, such that
\begin{equation}
\| f^\varepsilon \|_{L^2((0,T)\times\omega)} \leqslant C(T) \| u^0 \|_{(L^2(0,1))^n}.
\end{equation}
\end{theorem}

\begin{theorem}[Homogenization] \label{can_one_homo}
Assume that rank$[b|Ab|A^2b |\ldots |A^{n-1}b] = n$ and that $A$ satisfies (H2). Then, there exists a sequence of null control functions $\{f_\varepsilon\}_{\varepsilon\in (0,1)}$ to system \eqref{canonical_system}, such that 
$$f_\varepsilon\to f_0 \text{ strongly in } L^2((0,T)\times\omega), $$
where $f_0$ is a  null control to the homogenized system of \eqref{canonical_system}, given by
\begin{align}\label{eq_homogen_one}
\begin{dcases}
\p_t {u}-\mc{L}^0 {u} +\mathcal{C} {u}= e_11_{\omega} f_0, & (t,x)\in (0, T)\times (0,1), \\ 
{u}(t,0)={u}(t,1)=0, & t \in (0,T),\\
{u} (0,x)={u}^0& x \in (0,1),
\end{dcases}
\end{align}
where the homogenized operator $\mathcal{L}^0 := \left(\left(\ds \int_0^1\frac{1}{a(s)}ds\right)^{-1}\partial_{xx} \right) I_{n\times n}.$
\end{theorem}

\begin{rem} \label{rem_equiv}
Due to the equivalence mentioned in \cref{app_mcontrol}, a positive result for the controllability of \eqref{general_C_M}-\eqref{canonical_system} also implies controllability for the pair \eqref{general_B_M}-\eqref{canonical_system_n}. Analogously, a positive homogenization result for system \eqref{eq_homogen_one} implies the same for system \eqref{eq_homogen}. 
Thus, to prove that \cref{thm_control_m} and \cref{thm_homogen_m} hold, it is enough to prove \cref{thm_control_1} and \cref{can_one_homo}. 
\end{rem}

\section{Uniform Controllability: \cref{thm_control_1}} \label{sec_mainproof}

By the well-known duality argument \cite{dol_rus}, studying the controllability of \eqref{canonical_system} is equivalent to studying the observability problem for the following adjoint system
\begin{align} \label{adjointsystem_w}
\begin{dcases}
- \p_t w^\varepsilon- \mc{L}^\varepsilon w^\varepsilon+\mc{C}^* w^\varepsilon=0, & (t,x)\in (0, T)\times (0,1), \\ 
w^\varepsilon(t,0)=w^\varepsilon
(t,1)=0, & t \in (0,T),\\
w^\varepsilon(T,x)=w^\varepsilon_0(x), & x \in (0,1),
\end{dcases}
\end{align}
where $\mc{C}^*$ denotes the adjoint of the coupling matrix $\mc{C}$. Hence, to prove controllability of \eqref{canonical_system}, it is enough to obtain an observability estimate of the form
\[ \|w^\varepsilon(0,\cdot)\|_{(L^2(0,1))^n}\leqslant C(T) \int_{0}^T\int_{\omega}|w_1^\varepsilon|^2 \rd x \rd t, \]
for some positive constant $C(T)$, where $w^\varepsilon$ is the solution of the adjoint system \eqref{adjointsystem_w}. Henceforth, we will work to show an estimate of the above type.

In order to demonstrate our control strategy, we need information about the spectral behaviour of the operator $-\mathcal{L}^\varepsilon + \mc{C}^*$. For each $\varepsilon>0$, the spectrum will be divided into two parts, low frequency and high frequency. We will make precise the set of low frequencies and high frequencies at the end of the following section.

\subsection{Spectral Analysis} \label{ssec_stepi}
We will study the spectral problem for the operator $-\mathcal{L}^\varepsilon + \mc{C}^*$, given by 
\begin{align}\label{nd_spc_pro}
\begin{split}
-\mathcal{L}^\varepsilon\Phi^\varepsilon+\mathcal{C}^*\Phi^\varepsilon=\mu^\varepsilon \Phi^\varepsilon,&~~~\text{ in } (0,1), \\
\Phi^\varepsilon(0)=\Phi^\varepsilon(1)=0.&
\end{split}
\end{align}
In order to study the above eigenvalue problem, we recall a result from \cite{lopez_zuazua}, concerning the spectral gap for the following eigenvalue problem
\begin{align*} 
\begin{split}
-\partial_x(a^\varepsilon(x)\partial_x\varphi^\varepsilon)=\lambda^\varepsilon \varphi^\varepsilon, & \text{ in } (0,1),\\
\varphi_\varepsilon(0)=\varphi_\varepsilon(1)=0.&
\end{split}
\end{align*}
For each $\varepsilon>0,$  there exists a sequence of spectral pairs $\{(\lambda_k^\varepsilon, \varphi_k^\varepsilon)\}_{k\in \mathbb{N}}$ such that 
\[ 0<\lambda^\varepsilon_1<\lambda_2^\varepsilon<\cdots<\lambda_k^\varepsilon<\cdots \to \infty,\]
and $\{\varphi_k^\varepsilon\}_{k\in \mathbb{N}}$ form an orthonormal basis for $L^2(0,1)$. Also, by comparing with the spectrum of the operator $-\partial_{xx},$ it can be shown that $\lambda_k^\varepsilon\sim k^2.$ Using the above spectral pairs $\{(\varphi_k^\varepsilon,\lambda_k^\varepsilon)\}$, we will find the eigenvalues and eigenfunctions for the eigenvalue problem \eqref{nd_spc_pro}.
Let  $\varepsilon>0$ be fixed. Notice that, for each $k\in \mathbb{N},$ the $n$-dimensional sub space
\[V^\varepsilon_k=\text{span}\left\{\begin{pmatrix}
	\varphi^\varepsilon_k \\ 0 \\ \vdots \\ 0
 \end{pmatrix}, \begin{pmatrix}
	0 \\ \varphi^\varepsilon_k \\ \vdots \\ 0
 \end{pmatrix}, \ldots,
 \begin{pmatrix}
	0 \\ 0 \\ \vdots \\ \varphi^\varepsilon_k
 \end{pmatrix}
\right\}\]
is an invariant space for the operator $-\mathcal{L}^\varepsilon +C^*$. Furthermore, we also have
\[(L^2(0,1))^n= \oplus_{k\in \mathbb{N}} V^\varepsilon_k.\]
Hence, if we are able to show that, for each $k\in \mathbb{N}$, there are  $n$-linearly independent eigenvectors in the invariant space $V_k$, we  will be done.  

For a fixed $\varepsilon>0,$ consider the invariant space $V_k^\varepsilon$. Then, the spectral problem \eqref{nd_spc_pro} in the subspace $V_k^\varepsilon$ reduces to finding $(c_1,c_2,\cdots,c_n)\in \mathbb{R}^n$, such that
\begin{align*}
(-\mathcal{L}^\varepsilon+\mathcal{C}^*)\begin{pmatrix}
c_1 \varphi^\varepsilon_k & c_2 \varphi_k^\varepsilon & c_3 \varphi^\varepsilon_k &\cdots & c_n \varphi^\varepsilon _k
\end{pmatrix}^{tr}=\mu^\varepsilon_k \begin{pmatrix}
c_1 \varphi^\varepsilon_k & c_2 \varphi^\varepsilon_k & c_3 \varphi^\varepsilon_k &\cdots & c_n \varphi^\varepsilon _k
\end{pmatrix}^{tr}.
\end{align*}
Since $ \varphi_k^\varepsilon$ is an eigenfunction for $\partial_x(a^\varepsilon(x)\partial_x)$ corresponding to $\lambda_k^\varepsilon$, the above equation reduces to the following identity
\[
\begin{bmatrix}
\lambda^\varepsilon _k & 0 & 0 & \ldots & 0 \\
0 & \lambda^\varepsilon_k& 0 & \ldots & 0 \\ 
0 & 0 & \lambda^\varepsilon_k & \ddots & \vdots \\
\vdots & \vdots & \vdots & \ddots & \vdots\\
0 & 0 & 0 & \ldots & \lambda^\varepsilon_k
\end{bmatrix}
\begin{pmatrix}
c_1 \varphi^\varepsilon_k \\ c_2 \varphi^\varepsilon _k \\ c_3 \varphi^\varepsilon _k \\ \vdots \\ c_n \varphi^\varepsilon _k 
\end{pmatrix}
+ \begin{bmatrix}
0 & 1 & 0 & \ldots & 0 \\
0 & 0 & 1 & \ldots & 0 \\ 
0 & 0 & 0 & \ddots & \vdots \\
\vdots & \vdots & \vdots & \ddots & 1\\
a_1 & a_2 & a_3 & \ldots & a_n
\end{bmatrix} 
\begin{pmatrix}
c_1 \varphi_k^\varepsilon \\ c_2 \varphi_k^\varepsilon \\ c_3 \varphi_k^\varepsilon \\ \vdots \\ c_n \varphi_k^\varepsilon 
\end{pmatrix}
=
\mu_k^\varepsilon \begin{pmatrix}
c_1 \varphi_k^\varepsilon \\ c_2 \varphi_k^\varepsilon \\ c_3 \varphi_k^\varepsilon \\ \vdots \\ c_n \varphi_k^\varepsilon 
\end{pmatrix}.
\]
Now, let $\sigma$ be an eigenvalue of $\mc{C}^*$. Then, the above equation implies that $\sigma = \mu_k^\varepsilon - \lambda_k^\varepsilon$. Hence, writing the above as a system of linear equations, we get
\begin{align*}
-\sigma c_1 + c_2 & = 0, \\
-\sigma c_2 + c_3 & = 0, \\
\vdots \\
-\sigma c_{n-1} + c_n & = 0, \\
a_1 c_1 + \ldots + a_n c_n -\sigma c_n & = 0.
\end{align*}
A solution to the above system is given by $(c_1,c_2,\ldots,c_n)=\frac{\sigma^{m-1}}{\sqrt{\sum_{l=1}^{n-1} \sigma^{2(l-1)}}}(1,\sigma,\ldots,\sigma^{n-1})$, where $\sigma \in \text{Spec}(\mathcal{C}^*)$. Due to (H2), $A$ has $n$-distinct eigenvalues, and as $P^{-1} A P = \mc{C}$, we get that $\mc{C}$ has $n$-distinct eigenvalues. This implies that $\mc{C}^*$ has $n$-distinct eigenvalues as well, say $\sigma_1,\ldots,\sigma_n$. Then, we conclude that each $(\lambda^\varepsilon_k,\varphi^\varepsilon_k)$ gives rise to $n$-eigenpairs of the form 
\begin{align*}
\left\{ \left(\mu^\varepsilon_{k,i},\frac{\sigma_{i}^{m-1}}{\sqrt{\sum_{l=1}^{n-1} \sigma_i^{2(l-1)}}} \begin{pmatrix}
\varphi_k ^\varepsilon\\ \sigma_{i} \varphi_k^\varepsilon \\ \vdots \\ \sigma_{i}^{n-1} \varphi_k^\varepsilon
\end{pmatrix} \right) \right\}_{i=1,\ldots,n}=:\{(\mu_k^\varepsilon, \Phi_{k,i}^\varepsilon)\}_{i=1,2,\dots,n}.
\end{align*}
where $\mu^\varepsilon_{k,i}=\sigma_i+\lambda_k^\varepsilon.$  Due to (H2), $\mu_{k,i}^\varepsilon$ are strictly positive, for$ i=1,2,\dots,n,$ and $k\in \mathbb{N}$. Hence, we have shown that, for each $\varepsilon>0$, in each $V_k^\varepsilon$ there are $n$-linearly independent eigenvectors. We summarise this result in the following theorem.

\begin{theorem}
For each $\varepsilon\in(0,1),$  $\{(\mu_{k,i}^\varepsilon, \Phi_{k,i}^\varepsilon)\}_{(i=1,2,\dots,n;~~~ k \in \mathbb{N})}$ is a sequence  of eigenpair to the spectral problem \eqref{nd_spc_pro}, such that 
$\{\Phi_{k,i}^\varepsilon\}_{(i=1,2,\dots,n;~~~ k \in \mathbb{N})}$ forms an orthonormal basis for $(L^2(0,1))^n.$ 
\end{theorem}

With out loss of generality, we can relabel the sequence $\{(\mu_{k,i}^\varepsilon, \Phi_{k,i}^\varepsilon)\}_{(i=1,2,\dots,n;~~~ k \in \mathbb{N})}$ and put them in the increasing order  and arrange them as   $\left\{\mu_k^\varepsilon, \Phi_k^\varepsilon \right\}_{k=1}^{\infty}$.  Since $\lambda_k^\varepsilon\sim k^2$, we have $\mu_k^\varepsilon\sim k^2$. Next, we have the following  spectral gap and positivity theorem for the sequence $\{(\mu_{k}^\varepsilon,\Phi_{k}^\varepsilon)\}_{k=1}^{\infty}.$ The proof is analogous to the proof presented in \cite{lopez_zuazua}; a brief outline of the same is given in \cref{app_eigen}.

\begin{prop}\label{spectral_thm_couple} Assume that $a \in W^{2,\infty}(\mathbb{R})$ and satisfies \eqref{upper_lower_bound_1}. Also assume that (H1) and  (H2) are satisfied. Then given any $\delta>0$, there exists a constant $C(\delta)$, independent of $\varepsilon$, such that  
\[\sqrt{\mu^\varepsilon_{k+1}}-\sqrt{\mu^\varepsilon_k}\geqslant \frac{\pi}{\sqrt{\bar{\tilde{a}}}}-\delta,~~~\forall k \varepsilon\leqslant  C(\delta),\]
where $\bar{\tilde{a}}$ is a suitable transformation of the function $a$; see equation \eqref{avg-tilde-a} for the precise definition. Furthermore, there exist $C_1$ and $\tilde{C}>0$, such that
\begin{align*}
\|\Phi_{k}^{\varepsilon}\|_{(L^2(\omega)^n)}\geqslant \tilde{C},~~\forall k \in \mathbb{N}, \text{ satisfying }k\varepsilon \leqslant  C_1.
\end{align*}
\end{prop}
Now, for each $\varepsilon\in (0,1),$ we will define the low frequency and high frequency spaces.  Let  $D=\min\{{C_1, C(\frac{\pi}{2\sqrt{\bar{\tilde{a}}}})}\}$. Define the set of low frequencies as the set $\{\Phi_k^\varepsilon: k\leq[D /\varepsilon]\},$ and the set of high frequency as the set $\{\Phi_k^\varepsilon:k > [D/\varepsilon]\}.$ Let us denote the vector space spanned by the low frequencies by
\[H_{[D/\varepsilon]}=\text{span}\{\Phi_k^\varepsilon: k \leq [D/\varepsilon]\},\]
and let $\Pi$ denote the usual projection operator.

In the first time interval $[0,T/3]$, we will prove the uniform null controllability for the low frequencies for the canonical system \eqref{canonical_system}. That is, we will show that, for each $\varepsilon>0$, there is a control  such that $\Pi _{H_{[D/\varepsilon]}} u_\varepsilon(T/3)=0$, and the cost of the control is uniformly bounded.   
Let us recall the adjoint system restricted to the first time interval
\begin{align} \label{adjointsystem_ist_step}
\begin{dcases}
- \p_t w^\varepsilon- \mc{L}^\varepsilon w^\varepsilon+\mc{C}^*
w^\varepsilon =0, & (t,x)\in (0, T/3)\times (0,1), \\ 
w^\varepsilon (t,0)=w^\varepsilon (t,1)=0, & t \in \left( 0, T/3\right),\\
w(T/3,x)=w_0^{\varepsilon}(x), & x \in (0,1).
\end{dcases}
\end{align}
By duality, it is enough to prove that the observability constant is independent of $\varepsilon$  for any $w_0^\varepsilon\in H_{[D/\varepsilon]}$. In particular we need to show the following result;
\begin{prop}
Let $w^\varepsilon =(w_1^\varepsilon,w_2^\varepsilon,\cdots, w_n^\varepsilon )^\text{tr}$ is the solution to the adjoint system \eqref{adjointsystem_ist_step}  with given data $w^0_\varepsilon \in H_{[D/\varepsilon]}.$ Then, given any $T>0,$ there exists a positive constant $C(T)$ independent of $\varepsilon,$ such that 
\begin{align} \label{observ_1st}
\int_0^{T/3}\int_\omega|w_1^\varepsilon|^2 \rd x \rd t \geqslant C(T)\|w^\varepsilon(0,\cdot)\|_{(L^2(0,1))^n}^2.
\end{align} 
\end{prop}

\begin{proof}
Since $w_0^\varepsilon \in H_{[D/\varepsilon]}$, we can express $w_0^\varepsilon$ as
\[w_0^\varepsilon=\sum_{k\leqslant [D/\varepsilon]}\alpha_k\Phi_k^\varepsilon, \quad \alpha_k\in \mathbb{R}. \]
Then $w^{\varepsilon},$ the solution of \eqref{adjointsystem_ist_step}, is given by
\[w^\varepsilon(t,x)=\begin{pmatrix}
w_1^\varepsilon\\ w_2^\varepsilon\\ \vdots\\ w_n^\varepsilon
\end{pmatrix}=\sum_{k\leqslant [D/\varepsilon]}\alpha_k  e^{-\mu_k^\varepsilon(T/3-t)}\Phi_k^\varepsilon(x) =\sum_{k\leqslant [D/\varepsilon]}\alpha_k  e^{-\mu_k^\varepsilon (T/3-t)}\begin{pmatrix}
c_{1k}\phi_k^\varepsilon\\ c_{2k} \phi_k^\varepsilon\\ \vdots \\c_{nk} \phi_k^\varepsilon
\end{pmatrix},\]
with appropriate constants $c_{i,k}$ for $i\in\{1,2,\dots,n\}$ and $$w^\varepsilon(0,t)=\sum_{k\leqslant [D/\varepsilon]}\alpha_k  e^{-\mu_k^\varepsilon T/3}\begin{pmatrix}
c_{1k}\phi_k^\varepsilon\\ c_{2k} \phi_k^\varepsilon\\ \vdots \\c_{nk} \phi_k^\varepsilon
\end{pmatrix}.$$
Hence, we want to prove the following inequality
\begin{align*}
\int_{0}^{T/3}\int_{\omega} |w_1^\varepsilon|^2 \rd x \rd t\geqslant C(T)\sum_{k\leqslant [D/\varepsilon]}\alpha_k^2  e^{-2\mu_k^\varepsilon T/3}.
\end{align*}
Let us calculate the left hand side of the above inequality
\begin{align*}
\int_0^{T/3} \|w_1^\varepsilon\|^2_{L^2(\omega)} \rd t & = \int_0^{T/3} \bigg\|\sum_{k\leqslant [D/\varepsilon]} \alpha_k c_{1k}\varphi_k^\varepsilon  e^{-\mu_k^\varepsilon (T/3-t)} \bigg\|_{L^2(\omega)}^2 \rd t \\
& = \int_0^{T/3}\bigg\|\sum_{k\leqslant [D/\varepsilon]} \alpha_kc_{1k}\varphi_k^\varepsilon  e^{-\mu_k^\varepsilon t}\bigg\|_{L^2(\omega)}^2 \rd t.
\end{align*}
Let $\{\psi_j\}_{j\in \mathbb{N}}$ be an orthonormal basis for $L^2(\omega)$, then
\begin{align*}
\int_0^{T/3} \bigg\|\sum_{k\leqslant [D/\varepsilon]} \alpha_k c_{1k}\varphi_k^\varepsilon  e^{-\mu_k^\varepsilon t}\bigg\|_{L^2(\omega)}^2 \rd t = \int_0^{T/3}\sum_{j=1}^\infty \bigg|\sum_{k\leqslant [D/\varepsilon]}\alpha_k c_{1k} \langle \varphi_k^\varepsilon, \psi_j \rangle_{L^2(\omega)} e^{-\mu _k^\varepsilon t} \bigg|^2 \rd t.
\end{align*}
Due to \Cref{spectral_thm_couple}, using \cite[Corollary 3.6]{tenenbaum_tucsnak}, we have the existence of $M_1$ and $M_2$, independent of $\varepsilon$, such that
$$M_1 e^{M_2/T} \int_0^{T/3} \bigg|\sum_{k=0}^\infty b_k e^{-\mu_k^\varepsilon t}\bigg|^2 \rd t\geqslant \sum_{k=0}^\infty b_k^2 e^{-2\mu_k^\varepsilon T/3},$$
for any $\{b_k\}\in l^2(\mathbb{R})$.
Let us choose $b_k$ as follows
\begin{align*}
b_k := \begin{cases}
		\alpha_k c_{1k} \langle \varphi_k^\varepsilon, \psi_j \rangle, & k \leqslant [D/\varepsilon], \\
		0, & k>[D/\varepsilon].
		\end{cases}
\end{align*}
Then we have
\begin{multline*}
M_1 e^{M_2/T} \int_0^{T/3}\|w_1^\varepsilon\|^2_{L^2(\omega)} \rd t \geqslant \sum_{j=0}^\infty\sum_{k\leqslant [D/\varepsilon]} e^{-2\mu_k^\varepsilon T/3}\alpha_k^2{c_{1k}}^2\langle \varphi_k^\varepsilon, \psi_j\rangle^2_{L^2(\omega)} \\
= \sum_{k\leqslant [D/\varepsilon]}  e^{-2\mu_k^\varepsilon T/3} \alpha_k^2 c_{1k}^2 \sum_{j=0}^{\infty} \langle \varphi_k^\varepsilon, \psi_j\rangle^2_{L^2(\omega)} = \sum_{k\leqslant [D/\varepsilon]}  e^{-2\mu_k^\varepsilon T/3} \alpha_k^2c_{1k}^2 \|\varphi_k^\varepsilon\|_{L^2(\omega)}^2 \\
\geqslant C(T) \sum_{k\leqslant [D/\varepsilon]}  e^{-2\mu_k^\varepsilon T/3} \alpha_k^2c_{1k}^2 \geqslant C(T) \min_{{k\leqslant [D/\varepsilon]}}\left\{c_{1k}^2\right\}   \sum_{k\leqslant [D/\varepsilon]}e^{-2\mu_k^\varepsilon T/3} \alpha_k^2.
\end{multline*}
Note that,  for  $ k\leqslant [D/\varepsilon] $, from the expression of eigenvector, 
$\ds{c_{1k}^2}$ takes values in the following set
$$\mathbb{E}^{\sigma}:=\left\{\frac{1}{ \sum_{i=0}^{n-1}\sigma^{2i}}:~\sigma \in \text{Spec} (A):=\{\sigma_1,\sigma_2,\cdots, \sigma_n\}\right\}.$$
Hence, the above inequality reduces to the desired estimate
\[ \int_0^{T/3}\int_{\omega}\|w_1^\varepsilon\|^2 \rd x \rd t \geqslant  C(T) \min\{\mathbb{E}^\sigma\} \sum_{k\leqslant [D/\varepsilon]} \alpha_k^2 e^{-2\mu_k^\varepsilon T/3}. \qedhere \]
\end{proof}
\begin{rem}\label{tby3bound}
From the above proposition, we conclude that for $\varepsilon>0,$ there exits a control  $f_1^\varepsilon \in L^2(0,1)$ to the system \eqref{canonical_system}  such that  $\Pi_{H_{[D/\varepsilon]}}u^\varepsilon(\cdot, T/3)=0$. The uniform observability constant in \eqref{observ_1st} guarantees that
$$\|f^\varepsilon_1\|_{L^2((0,T/3)\times \omega)} \leqslant C(T)\|u_0\|.$$
\end{rem}

\subsection{Decay} \label{ssec_stepii}
In this part of the solution, we do not put any control and let the solution decay. The system evolves freely on the time interval $(T/3,2T/3)$, with initial data given by the solution from the first step, at time $T/3$. Let us write $z_0^\varepsilon(x) := u^\varepsilon(T/3,x)$. Note that, $z_0^\varepsilon \in H_{[D/\varepsilon]}^\perp.$ The second step is nothing but proving the following proposition.
\begin{prop}
Let $u^\varepsilon$ be the solution to the following system
\begin{align*} 
\begin{dcases}
\p_t u^\varepsilon-\mc{L}^\varepsilon u^\varepsilon + \mc{C} u^\varepsilon=0, \quad & (t,x)\in (T/3,2T/3)\times (0,1), \\ 
u^\varepsilon (t,0) = u^\varepsilon (t,1)=0, & t \in (T/3,2T/3),\\
u_\varepsilon (T/3,x)=z^\varepsilon_0, & x \in (0,1),
\end{dcases}
\end{align*}
where $z^0_\varepsilon \in H_{[D/\varepsilon]}^\perp.$
Then, for any $t>0$, the following estimate holds
\begin{align*}
\|u^\varepsilon(t)\|^2_{(L^2(0,1))^2}\leqslant  e^{-2\mu_{[D/\varepsilon]}(t-T/3)} \|z^0_\varepsilon\|_{(L^2(0,1)}.
\end{align*}   
\end{prop}
\begin{proof}
Since $z^0_\varepsilon\in H_{[D/\varepsilon]}^\perp$,  we can write $z^0_\varepsilon$ as
\[z^0_\varepsilon(x)=\sum_{k>[D/\varepsilon]}\alpha_k\Phi_k^\varepsilon(x).\]
Then the solution $u^\varepsilon$  can be written as
\[u^\varepsilon(t,x)=\sum_{k>[D/\varepsilon]}\alpha_k  e^{-\mu_k^\varepsilon (t-T/3)}\Phi_k^\varepsilon(x).\]
Finally, we estimate $\| u^\varepsilon(t) \|_{(L^2(0,1))^n}$ as follows: for $t\in (T/3,2T/3)$
\[ \|u^{\varepsilon}(t)\|^2_{(L^2(0,1))^n}= \sum_{k>[D/\varepsilon]}\alpha_k^2 e^{-2\mu_k^\varepsilon (t-T/3)}\leqslant  e^{-2\mu_{[D/\varepsilon]}(t-T/3)}\sum_{k>[D/\varepsilon]}\alpha_k^2=  e^{-2\mu_{[D/\varepsilon]}(t-T/3)}\|z^0_\varepsilon\|_{(L^2(0,1))^n},\]
which completes the proof.
\end{proof}

\subsection{Observability via Carelman estimate} \label{ssec_stepiii}
For the time interval $(2T/3,T)$, we will prove an observability result for the function $w$, solution of adjoint system \eqref{adjointsystem_w}, by using a parabolic Carleman estimate along with a standard energy estimate.
The first step in this direction is to remove the $a(\cdot)$ coefficient from the second order term in \eqref{adjointsystem_w}, for which we use a coordinate transformation inspired from \cite{lopez_zuazua}. We rewrite the system using the following change of variable
\begin{equation} \label{eq_change_of_var}
z := \int_0^{\frac{x}{\varepsilon}} \frac{1}{a(s)} \rd s \cdot \left( \int_0^{\frac{1}{\varepsilon} }\frac{1}{a(s)} \rd s\right)^{-1}. 
\end{equation}
Then, system  \eqref{adjointsystem_w} transforms to the following system
\begin{align}\label{1st_adjointsystem}
\begin{dcases}
- \p_t w-\frac{1}{\tilde{a}\left(\frac{z}{\varepsilon}\right)} \p_{zz} w+\mc{C}^*w=0, & (t,z) \in (0,T) \times (0,1),\\
w(t,0)=w(t,1)=0, & t \in (0,T),\\
w(T,z)=\hat{w}(z), & z \in (0,1),
\end{dcases}
\end{align}
where $w=(w_1,\ldots, w_n)^\text{tr}$, $\hat{w}(z)=w^0 \left(\varepsilon h^{-1} \left(\frac{z}{\delta(\varepsilon)} \right)\right)$, and $\tilde{a}$ is an appropriate transformation of $a$, defined in \eqref{atilde}.
As we will derive the estimate for a fixed
$\varepsilon,$ we write $\tilde{a}\left(\frac{z}{\varepsilon}\right)=r(z)$. Next, applying the following change of variable
\begin{equation} \label{eq_w_tildew}
y:=H(z)=\int_0^{z}\sqrt{r(s)} \rd s,
\end{equation}
system \eqref{1st_adjointsystem} reduces to the following
\begin{align}\label{2nd_adjoint_system}
\begin{dcases}
-\p_t  \tilde{w}- \p_{yy} \tilde{w} - g \p_y \tilde{w} + \mc{C}^* \tilde{w}=0, & (t,y) \in (0,T) \times (0,L),\\
\tilde{w}(t,0)=\tilde{w}(t,L)=0, & t \in (0,T),\\
\tilde{w}(y,T) = \check{w}(y) = \hat{w}(H^{-1}(y)), & y \in (0,L),
\end{dcases}
\end{align}
where $\ds L=\int_0^1\sqrt{r(s)} \rd s$, and $\ds  g(y)=\frac{r'(H^{-1}(y))}{2r(H^{-1}(y))^{\frac{3}{2}}}$. To transform the $\p_y \tilde{w}$ term into a zeroth order term, we further reduce the above system, using integrating factor as follows
\begin{equation} \label{eq_cov_wtov}
v(y,t)= e^{ -\ds\int_{0}^y \frac{g(s)}{2} \rd s } \tilde{w}(y,t).
\end{equation}
Then system \eqref{2nd_adjoint_system} transforms into
\begin{equation} \label{eq_pde_v}
\begin{dcases}
-\p_t v- \p_{yy}v+b(y)v+\mc{C}^*v=0, & (t,y) \in (0,T) \times (0,L),\\
v(t,0)=v(t,L)=0, & t \in (0,T),\\
v(0,y)=v^0(y), & y \in (0,L),
\end{dcases}
\end{equation}  
where $\ds v^0(y)= e^{-\ds\int_{0}^y \frac{g(s)}{2}\rd s} \check{w}(y)$ and the function $b(y)$ is given by 
\[b(y) := \left( \frac{a'(H^{-1}(y))}{4 a(H^{-1}(y))^\frac{3}{2}} \right)' + \left( \frac{a'(H^{-1}(y))}{4a(H^{-1}(y))^{\frac{3}{2}}} \right)^2.\]
Note that, we define the functions $\tilde{w}$ and $v$ on the whole time interval $(0,T)$, even though we are only interested in the region $(2T/3,T )$. This is done purely for ease of presentation and does not affect the analysis.

First, we will derive an observability estimate for the function $v := (v_1,\ldots,v_n)^\text{tr}$. For computational convenience, we rewrite equation \eqref{eq_pde_v} explicitly as follows
\begin{align*} \numberthis
\label{eq_carl_sys_nxn}
- \p_t \begin{pmatrix}
v_1 \\ v_2 \\ v_3 \\ \vdots \\ v_n 
\end{pmatrix}
- \p_{yy} \begin{pmatrix}
v_1 \\ v_2 \\ v_3 \\ \vdots \\ v_n
\end{pmatrix}
& + b(y) \begin{pmatrix}
v_1 \\ v_2 \\ v_3 \\ \vdots \\ v_n
\end{pmatrix}
+
\begin{bmatrix}
0 & 1 & 0 & \ldots & 0 \\
0 & 0 & 1 & \ldots & 0 \\ 
0 & 0 & 0 & \ddots & \vdots \\
\vdots & \vdots & \vdots & \ddots & 1\\
a_1 & a_2 & a_3 & \ldots & a_n
\end{bmatrix} 
\begin{pmatrix}
v_1 \\ v_2 \\ v_3 \\ \vdots \\ v_n
\end{pmatrix}
= \begin{pmatrix}
0 \\ 0 \\ 0 \\ \vdots \\ 0
\end{pmatrix}. 
\end{align*}
For an operator of the above type, we have the following parabolic Carleman estimate \cite{coron_book,fur_iman,tucsnak_weiss}, which we present without proof; see the cited references for the proof.

\begin{theorem}[Carleman estimate] \label{thm_para_carl_est}
Let $T>0$ be fixed. Let $\theta : (0,T) \rightarrow \mb{R}$ be the function given by
\[\theta(t) = \frac{1}{t(T-t)}.\]
Let $\omega \subset (0,L)$ be a non-empty open subset. Then, there exists a function $\psi \in C^2[0,L]$ such that
\[ \sup_{(0,L)} \psi <0, \qquad \inf_{(0,L) \setminus \omega} |\nabla \psi| >0,\]
and for any $ d \in \mb{R}$, there exist constants $s_0>0$ and $C>0$ such that the following holds
\begin{align*}
\int_0^T \int_0^L(s \theta)^d \left| e^{s\theta\psi} v \right|^2 + \int_0^T \int_0^L(s \theta)^{d-2} \left| e^{s\theta\psi} \p_y v \right|^2 & + \int_0^T \int_0^L(s \theta)^{d-4} \left| e^{s\theta\psi} \p_t v \right|^2 \numberthis \label{eq_para_carl_est}\\
+ \int_0^T \int_0^L(s \theta)^{d-4}  \left| e^{s\theta\psi} \p_{yy} v \right|^2 \leqslant C \big( \int_0^T \int_\omega(s \theta)^d & \left| e^{s\theta\psi} v \right|^2 + \int_0^T \int_0^L(s \theta)^{d-3} \left| e^{s\theta\psi} (\p_t v \pm \p_{yy} v) \right|^2 \big),
\end{align*}
for any $s \geqslant s_0$ and any $v \in C^2([0,T]\times [0,L])$ with $v(t,0)= v(t,L)=0$ on $(0,T)$.
\end{theorem}

Note that, in the above statement we have taken $\omega$ as the integral region in the first term on the RHS. However, to be precise, it should be for some $\tilde{\omega} \subset (0,L)$. This is done to maintain convenient notation and can be done without loss of generality. We can get back $\omega$ from $\tilde{\omega}$ when we do the reverse coordinate transformation related to \eqref{eq_cov_wtov}.


Before proving the required observability estimate, we provide the following energy estimate for solutions of \eqref{eq_pde_v}, that can be obtained using standard energy arguments. 
\begin{prop} \label{prop_energy}
Let $v$ solve \eqref{eq_pde_v}. Then there exists $C'>0$, such that
\begin{equation} \label{eq_energy}
\| v_1 (t) \|_{L^2(0,L)}^2 + \ldots + \| v_n (t) \|_{L^2(0,L)}^2 \leqslant e^{C'(\tau-t)} \left( \| v_1(\tau) \|_{L^2(0,L)}^2 + \ldots + \| v_n(\tau) \|_{L^2(0,L)}^2 \right),
\end{equation}
for all $0 \leqslant t \leqslant \tau \leqslant T$.
\end{prop}
Then we have the following observability estimate.

\begin{prop}
There exist constants $C, C', C_3>0$ such that the following is satisfied
\begin{equation} \label{eq_obs_est_v_n}
\left\| v_1 \left(\frac{2T}{3} \right) \right\|_{L^2(0,L)}^2 + \ldots + \left\| v_n \left(\frac{2T}{3} \right) \right\|_{L^2(0,L)}^2 \leqslant e^{C_3 (s_2 + C \| b \|_{L^\infty}^{\frac{2}{3}})} e^{\frac{2C'T}{3}} \frac{12}{T} \int_0^T \int_\omega v_1^2 \rd y \rd t,
\end{equation}
for every solution $v$ of \eqref{eq_pde_v}.
\end{prop}

To prove the above observability result, we will use \cref{thm_para_carl_est}. The proof is inspired from \cite{boyer}, where the author considers two equations with zeroth order coupling and presents a method to obtain observability with only one component. For our problem, we have $n$-equations with zeroth order coupling and an extra term in the equation \eqref{eq_carl_sys_nxn}, given by $b(y) v$. As the proof is technical, we provide an outline here. 

\begin{itemize}

\item[Step A.] First, we use \eqref{eq_para_carl_est} for each of the components $v_i,\ i=1,\ldots,n$. This gives us terms containing $v_1,\ldots,v_n$ over the observed region $\omega$ as well as the bulk terms with the operator $(\p_t + \p_{yy})$ on the RHS, while the LHS contains weighted energy type integrals for $v_1,\ldots,v_n$. Then we make specific choices for the parameter $d$, present in \cref{thm_para_carl_est}, for each component $v_i$, and add the component-wise estimates. The choice of the parameter $d_i$ for each $v_i$ is crucial for absorbing the non-essential terms.

\item[Step B.] Evaluating the $(\p_t + \p_{yy})$ term using \eqref{eq_carl_sys_nxn} results in the estimates containing the $b(y) v$ and zeroth-order terms from the coupling matrix. Next, we absorb these terms into the LHS appropriately, by choosing a large enough parameter $s$ provided by \cref{thm_para_carl_est}.

\item[Step C.] At this stage, on the RHS, we are only left with integrals with $v_i,\ i=1,\ldots,n$, over the region $\omega$. Our next goal is to bound the $v_2,\ldots,v_n$ terms by only $v_1$. This is done in a recursive manner---first we estimate the $v_n$ term by a $v_{n-1}$ term, then the $v_{n-1}$ term by a $v_{n-2}$ term, and continue till the last step, where we estimate the $v_2$ term by a $v_1$ term. This step crucially uses the fact that the coupling matrix in \eqref{eq_carl_sys_nxn} is cascade-type, due to which we can write $v_n$ exclusively in terms of $v_{n-1}$, and so on and so forth.

\item[Step D.] Now that we have the appropriate weighted integrals on both the sides, we estimate the Carleman weights from above and below. Finally, using the energy estimate from \eqref{prop_energy}, we obtain the observability result \eqref{eq_obs_est_v_n}.

\end{itemize}


\begin{proof}
For convenience, we will use the following notation 
\begin{equation} \label{eq_def_JI}
J(d,f) := \int_0^T \int_0^L (s \theta)^d \left| e^{s \theta \psi} f \right|^2 \rd y \rd t, \qquad I(d,f) := \int_0^T \int_\omega (s \theta)^d \left| e^{s \theta \psi} f \right|^2 \rd y \rd t.
\end{equation}

\noindent \underline{\textbf{Step A}}: First, we use \cref{thm_para_carl_est} to get component-wise estimates for $v_1, \ldots, v_n$. We will drop some terms from the LHS of the estimate for $v_1$ as those are not needed to show observability. In particular, using \eqref{eq_para_carl_est} gives us
\begin{align*}
&\bullet J(d_1,v_1) + J(d_1-2,\p_y v_1) \leqslant C I(d_1,v_1) + C J(d_1 -3, \p_t v_1 + \p_{yy} v_1), \\
&\bullet J(d_2,v_2) + J(d_2-2, \p_y v_2) + J(d_2-4,\p_t v_2) + J(d_2-4,\p_{yy} v_2)\\ &\hspace{6cm} \leqslant C I(d_2,v_2) + C J(d_2-3,\p_t v_2 + \p_{yy} v_2), \\
&\hspace{6.5cm} \vdots \\
&\bullet J(d_{n-1},v_{n-1}) + J(d_{n-1}-2, \p_y v_{n-1}) + J(d_{n-1}-4,\p_t v_{n-1}) + J(d_{n-1}-4,\p_{yy} v_{n-1}) \\
&\hspace{6cm}\leqslant C I(d_{n-1},v_{n-1}) + C J(d_{n-1}-3, \p_t v_{n-1} + \p_{yy} v_{n-1}), \\
&\bullet J(d_n,v_n) + J(d_n-2, \p_y v_n) + J(d_n-4,\p_t v_n) + J(d_n-4,\p_{yy} v_n)\\ & \hspace{6cm} \leqslant C I(d_n,v_n) + C J(d_n-3,\p_t v_n + \p_{yy} v_n).
\end{align*} 
Then we use \eqref{eq_carl_sys_nxn} to evaluate the $(\p_t + \p_{yy})$ terms on the RHS, to get
\begin{align*}
&\bullet J(d_1,v_1) + J(d_1-2,\p_y v_1) \leqslant C I(d_1,v_1) + C J(d_1 -3, b v_1 + v_2), \\
&\bullet J(d_2,v_2) + J(d_2-2, \p_y v_2) + J(d_2-4,\p_t v_2) + J(d_2-4,\p_{yy} v_2) \leqslant C I(d_2,v_2) + C J(d_2-3, b v_2 + v_3), \\
&\hspace{7cm} \vdots \\
&\bullet J(d_{n-1},v_{n-1}) + J(d_{n-1}-2, \p_y v_{n-1}) + J(d_{n-1}-4,\p_t v_{n-1}) + J(d_{n-1}-4,\p_{yy} v_{n-1}) \\
&\hspace{4.5cm}\leqslant C I(d_{n-1},v_{n-1}) + C J(d_{n-1}-3, b v_{n-1} + v_n), \\
&\bullet J(d_n,v_n) + J(d_n-2, \p_y v_n) + J(d_n-4,\p_t v_n)  + J(d_n-4,\p_{yy} v_n) \\& \hspace{4.5cm}\leqslant C I(d_n,v_n) + C J(d_n-3, b v_n + \sum_{i=1}^n a_i v_i).
\end{align*}
Now we will make a choice for the $d_i$'s that will allow us to do the necessary absorption of terms on the RHS. We set $d_k = 1 + 3(n+1 - k)$, for each $k=1,\ldots,n$. Then the latest estimates imply that
\begin{align*}
&\bullet J(3n+1,v_1) + J(3n-1,\p_y v_1) \leqslant C I(3n+1,v_1) + C J(3n-2, b v_1 + v_2), \\
&\bullet J(3n-2,v_2) + J(3n-4, \p_y v_2) + J(3n-6,\p_t v_2) + J(3n-6,\p_{yy} v_2)\\
&\hspace{4cm} \leqslant C I(3n-2,v_2) + C J(3n-5, b v_2 + v_3), \\
& \hspace{6cm}\vdots \\
&\bullet J(7,v_{n-1}) + J(5, \p_y v_{n-1}) + J(4,\p_t v_{n-1}) + J(4,\p_{yy} v_{n-1}) \leqslant C I(7,v_{n-1}) + C J(4, b v_{n-1} + v_n), \\
&\bullet J(4,v_n) + J(2, \p_y v_n) + J(0,\p_t v_n) + J(0,\p_{yy} v_n) \leqslant C I(4,v_n) + C J(1, b v_n + \sum_{i=1}^n a_i v_i).
\end{align*}

\noindent \underline{\textbf{Step B}}: Next, we will add the above individual estimates and then absorb the $J(\cdot,bv_i+v_{i+1})$ into the LHS. For this purpose, we consider the following algebraic expression
\[ \sum_{i=1}^n \epsilon^{n-i} \cdot (\text{estimate for } v_i), \quad \epsilon > 0,\]
to get the following estimate
\begin{multline*}
\epsilon^{n-1} J(3n+1,v_1) + \epsilon^{n-1} J(3n-1,\p_y v_1) + \epsilon^{n-2} J(3n-2,v_2) + \epsilon^{n-2} J(3n-4, \p_y v_2) + \epsilon^{n-2} J(3n-6,\p_t v_2) \\
+ \epsilon^{n-2} J(3n-6,\p_{yy} v_2) + \ldots + \epsilon J(7,v_{n-1}) + \epsilon J(5, \p_y v_{n-1}) + \epsilon J(4,\p_t v_{n-1}) + \epsilon J(4,\p_{yy} v_{n-1}) \\
+ J(4,v_n) + J(2, \p_y v_n) + J(0,\p_t v_n) + J(0,\p_{yy} v_n) \leqslant C \epsilon^{n-1} I(3n+1,v_1) + C \epsilon^{n-1} J(3n-2, b v_1 + v_2) \\
+ C \epsilon^{n-2} I(3n-2,v_2) + \epsilon^{n-2} C J(3n-5, b v_2 + v_3) + \ldots + C \epsilon I(7,v_{n-1}) + C \epsilon J(4, b v_{n-1} + v_n) + C I(4,v_n) \\
+ C J(1, b v_n + \sum_{i=1}^n a_i v_i).
\end{multline*}
Splitting the $J(\cdot,\cdot)$ integrals on the RHS, we get
\begin{multline*}
\epsilon^{n-1} J(3n+1,v_1) + \epsilon^{n-1} J(3n-1,\p_y v_1) + \epsilon^{n-2} J(3n-2,v_2) + \epsilon^{n-2} J(3n-4, \p_y v_2) + \epsilon^{n-2} J(3n-6,\p_t v_2) \\
+ \epsilon^{n-2} J(3n-6,\p_{yy} v_2) + \ldots + \epsilon J(7,v_{n-1}) + \epsilon J(5, \p_y v_{n-1}) + \epsilon J(4,\p_t v_{n-1}) + \epsilon J(4,\p_{yy} v_{n-1}) \\
+ J(4,v_n) + J(2, \p_y v_n) + J(0,\p_t v_n) + J(0,\p_{yy} v_n) \leqslant C \epsilon^{n-1} I(3n+1,v_1) + C \epsilon^{n-1} J(3n-2, b v_1) \\
+ C \epsilon^{n-1} J(3n-2, v_2) + C \epsilon^{n-2} I(3n-2,v_2) + \epsilon^{n-2} C J(3n-5, b v_2)  + \epsilon^{n-2} C J(3n-5, v_3) \\
+ \ldots + C \epsilon I(7,v_{n-1}) + C \epsilon J(4, b v_{n-1}) + C \epsilon J(4, v_n) + C I(4,v_n) + C J(1, b v_n + \sum_{i=1}^n a_i v_i).
\end{multline*}
For sufficiently small $\epsilon$, we can absorb $C \epsilon^{n-1} J(3n-2, v_2)$ on the RHS into $\epsilon^{n-2} J(3n-2,v_2)$ on the LHS, and do the same for $v_3,\ldots,v_n$, ending with absorbing $C \epsilon J(4, v_n)$ into $J(4,v_n)$. Then we remove the powers of $\epsilon$ and conclude that
\begin{multline*}
J(3n+1,v_1) + J(3n-1,\p_y v_1) + J(3n-2,v_2) + J(3n-4, \p_y v_2) + J(3n-6,\p_t v_2) 
+ J(3n-6,\p_{yy} v_2) \\ + \ldots + J(7,v_{n-1}) + J(5, \p_y v_{n-1}) + J(4,\p_t v_{n-1}) + J(4,\p_{yy} v_{n-1}) + J(4,v_n) + J(2, \p_y v_n) \\
+ J(0,\p_t v_n) + J(0,\p_{yy} v_n) \leqslant C  I(3n+1,v_1) + C J(3n-2, b v_1) 
+ C I(3n-2,v_2) + C J(3n-5, b v_2) \\
+ \ldots + C I(7,v_{n-1}) + J(4, b v_{n-1}) + C I(4,v_n) + C J(1, b v_n) + C J(1, \sum_{i=1}^n a_i v_i).
\end{multline*}
Next, we will consider the last term on the RHS of the above estimate and absorb this into the LHS. For this purpose, we see that, for some constant $\t C >0$, as
\[ (s\theta)^1 \leqslant \frac{\t C}{s^{3n}} (s\theta)^{3n+1}, \quad (s\theta)^1 \leqslant \frac{\t C}{s^{3n-3}} (s\theta)^{3n-2}, \ldots, (s\theta)^1 \leqslant \frac{\t C}{s^6} (s\theta)^7, \quad (s\theta)^1 \leqslant \frac{\t C}{s^3} (s\theta)^4, \]
choosing large enough $s$, say $s \geqslant s_1$, we can absorb $C J(1, \sum_{i=1}^n a_i v_i)$ on the RHS into $J(3n+1,v_1) + J(3n-2,v_2) + \ldots + J(7,v_{n-1}) + J(4,v_n)$ on the LHS. Note that, we can take care of the $a_i$'s multiplied with $v_i$'s, by choosing an appropriate constant $\t C$. This implies that
\begin{multline*}
J(3n+1,v_1) + J(3n-1,\p_y v_1) + J(3n-2,v_2) + J(3n-4, \p_y v_2) + J(3n-6,\p_t v_2) 
+ J(3n-6,\p_{yy} v_2) \\ + \ldots + J(7,v_{n-1}) + J(5, \p_y v_{n-1}) + J(4,\p_t v_{n-1}) + J(4,\p_{yy} v_{n-1}) + J(4,v_n) + J(2, \p_y v_n) \\
+ J(0,\p_t v_n) + J(0,\p_{yy} v_n) \leqslant C  I(3n+1,v_1) + C J(3n-2, b v_1) 
+ C I(3n-2,v_2) + C J(3n-5, b v_2) \\
+ \ldots + C I(7,v_{n-1}) + J(4, b v_{n-1}) + C I(4,v_n) + C J(1, b v_n).
\end{multline*}
Now, we will absorb the $J(\cdot,bv_i)$ terms on the RHS into the LHS. These terms can be absorbed into the LHS, if we have
\[ (s\theta)^{3n-2} b^2 < (s\theta)^{3n+1}, \quad (s\theta)^{3n-5} b^2 < (s\theta)^{3n-2}, \ldots, (s\theta)^4 b^2 < (s\theta)^7, \quad (s\theta)^1 b^2 < (s\theta)^4, \]
which is equivalent to the condition
\[\bar{C} \| b \|_{L^\infty}^{\frac{2}{3}} < s,\]
for some constant $\bar{C}>0$. Hence, for all $s > s_1 + \bar{C} \| b \|_{L^\infty}^{\frac{2}{3}}$, we have
\begin{multline*}
J(3n+1,v_1) + J(3n-1,\p_y v_1) + J(3n-2,v_2) + J(3n-4, \p_y v_2) + J(3n-6,\p_t v_2) 
+ J(3n-6,\p_{yy} v_2) \\ + \ldots + J(7,v_{n-1}) + J(5, \p_y v_{n-1}) + J(4,\p_t v_{n-1}) + J(4,\p_{yy} v_{n-1}) + J(4,v_n) + J(2, \p_y v_n) + J(0,\p_t v_n) \\
+ J(0,\p_{yy} v_n) \leqslant C_1 I(3n+1,v_1) + C_2 I(3n-2,v_2) + \ldots + C_{n-1} I(7,v_{n-1}) + C_n I(4,v_n). \numberthis \label{eq_CE_nxn_A6}
\end{multline*}

\noindent \underline{\textbf{Step C}}: Now, on the RHS, we are left with integrals over the control region $\omega$. Since we only want to control with $v_1$, our goal is to estimate $I(\cdot,v_i)$ for $i=2,\ldots,n$, by $I(\cdot,v_1)$.

Note that, by considering a different function $\psi$ in \cref{thm_para_carl_est}, we can replace the observation region $\omega$ present in the $I(\cdot,\cdot)$ terms, by a subset $\omega_1$ satisfying $\bar{\omega}_1 \subset \omega$. This will lead to a change in notation, but we keep it the same as it does not affect the analysis. Then let $\eta$ be a cutoff function  compactly supported in $\omega$, such that $0\leqslant \eta \leqslant 1$ and $\eta =1$ in $\omega_1$. For $I(4,v_n)$, using \eqref{eq_carl_sys_nxn}, we estimate as follows
\begin{align*} 
I(4,v_n) = \int_0^T \int_{\omega_1} (s\theta)^4 \left| e^{s\theta \psi} v_n \right|^2 \rd x \rd t  \leqslant \int_0^T \int_\omega \eta (s\theta)^4 \left| e^{s\theta \psi} v_n \right|^2 \rd x \rd t \\ = \int_0^T \int_{\omega} \eta (s\theta)^4 e^{2s\theta \psi} (\partial_t v_{n-1} + \partial_{yy} v_{n-1} - b v_{n-1}) v_n. \numberthis \label{eq_CE_nxn_B0}
\end{align*}
We will look at the three integrands on the RHS individually. First, we consider
\begin{align*}
& \left| \int_0^T \int_{\omega} \eta (s\theta)^4 e^{2s\theta \psi} \partial_t v_{n-1} v_n \right| \\
& = \left| - \int_0^T \int_{\omega} \eta (s\theta)^4 e^{2s\theta \psi} \partial_t v_n v_{n-1} - \int_0^T \int_{\omega} \eta s^4 \theta^3 (4 \theta' + 2 s \theta' \theta \psi) e^{2s\theta \psi} v_{n-1} v_n \right|.
\end{align*}
Note that, $\theta' \lesssim \theta^2$. Then $s^4 \theta^3 (4 \theta' + 2 s \theta' \theta \psi) \lesssim s^4 \theta^3 (\theta^2 + s \theta^3) \lesssim s^5 \theta^6 \lesssim (s\theta)^6 $. Hence, the above estimate, alongwith the Cauchy-Schwarz inequality, shows that
\begin{align*}
\left| \int_0^T \int_{\omega} \eta (s\theta)^4 e^{2s\theta \psi} \partial_t v_{n-1} v_n \right|
& = \int_0^T \int_{\omega} \eta (s\theta)^4 e^{2s\theta \psi} |\partial_t v_n v_{n-1}| + \int_0^T \int_{\omega} \eta (s\theta)^6 e^{2s\theta \psi} |v_{n-1} v_n| \\
& \leqslant C J(0,\p_t v_n)^{\frac{1}{2}} I(8, v_{n-1})^{\frac{1}{2}} + C J(4, v_n)^{\frac{1}{2}} I(8,v_{n-1})^{\frac{1}{2}} \\
& \leqslant \frac{1}{4C_n} J(0,\p_t v_n) + \frac{1}{8C_n} J(4, v_n) + C I(8, v_{n-1}). \numberthis \label{eq_CE_nxn_B1}
\end{align*}
Next, we consider the second integrand on the RHS of \eqref{eq_CE_nxn_B0} as follows
\begin{align*}
& \left| \int_0^T \int_{\omega} \eta (s\theta)^4 e^{2s\theta \psi} \partial_{yy} v_{n-1} v_n \right| \\
& = \left| - \int_0^T \int_{\omega} \eta (s\theta)^4 e^{2s\theta \psi} \p_y v_{n-1} \p_y v_n - \int_0^T \int_{\omega} (s\theta)^4 e^{2s\theta \psi} (\p_y \eta + 2 s \theta \p_y \psi) \partial_y v_{n-1} v_n \right| \\
& = \bigg| \int_0^T \int_{\omega} \eta (s\theta)^4 e^{2s\theta \psi} v_{n-1} \partial_{yy} v_n + \int_0^T \int_{\omega} (s\theta)^4 e^{2s\theta \psi} (\p_y \eta + 2 s \theta \p_y \psi) v_{n-1} \partial_y v_n \\
& \qquad + \int_0^T \int_{\omega} (s\theta)^4 e^{2s\theta \psi} (\p_y \eta + 2 s \theta \p_y \psi) v_{n-1} \p_y v_n \\
& \qquad + \int_0^T \int_{\omega} (s\theta)^4 e^{2s\theta \psi} \left[\p_{yy} \eta + 2 s \theta \p_{yy} \psi + 2s\theta \p_y \psi \p_y \eta + 4 s^2 \theta^2 |\p_y \psi|^2 \right] v_{n-1} v_n \bigg| \\
& \leqslant C \int_0^T \int_{\omega} (s\theta)^4 e^{2s\theta \psi} |v_{n-1} \partial_{yy} v_n| + C \int_0^T \int_{\omega} (s\theta)^5 e^{2s\theta \psi} |v_{n-1} \partial_y v_n| + C \int_0^T \int_{\omega} (s\theta)^6 e^{2s\theta \psi} |v_{n-1} v_n| \\
& \leqslant C J(0,\p_{yy} v_n)^{\frac{1}{2}} I(8, v_{n-1})^{\frac{1}{2}} +  C J(2, \p_y v_n)^{\frac{1}{2}} I(8,v_{n-1})^{\frac{1}{2}} + C J(4, v_n)^{\frac{1}{2}} I(8,v_{n-1})^{\frac{1}{2}} \\
& \leqslant \frac{1}{4C_n} J(0,\p_{yy} v_n) + \frac{1}{4C_n} J(2,\p_y v_n) + \frac{1}{8 C_n} J(4, v_n) + C I(8, v_{n-1}). \numberthis \label{eq_CE_nxn_B2}
\end{align*}
Finally, for $s > s_1 + \bar{C} \| b \|_{L^\infty}^{\frac{2}{3}}$, we have
\begin{equation}
\left|-\int_0^T \int_{\omega} \eta (s\theta)^4 e^{2s\theta \psi} b v_{n-1} v_n \right| \leqslant  \frac{1}{8C_n} J(4, v_n) + C I(7, v_{n-1}). \label{eq_CE_nxn_B3}
\end{equation}
Combining \eqref{eq_CE_nxn_B0}-\eqref{eq_CE_nxn_B3}, we get
\[C_n I(4,v_n) \leqslant \frac{1}{4} J(0,\p_t v_n) + \frac{3}{8} J(4, v_n) + \frac{1}{4} J(0,\p_{yy} v_n) + \frac{1}{4} J(2,\p_y v_n) + C C_n I(8, v_{n-1}).\]
Then, we can absorb the terms on the RHS in the above estimate by the terms on the LHS of \eqref{eq_CE_nxn_A6}. Indeed, substituting the above in \eqref{eq_CE_nxn_A6}, shows that
\begin{multline*}
J(3n+1,v_1) + J(3n-1,\p_y v_1) + J(3n-2,v_2) + J(3n-4, \p_y v_2) + J(3n-6,\p_t v_2) 
+ J(3n-6,\p_{yy} v_2) \\ + \ldots + J(7,v_{n-1}) + J(5, \p_y v_{n-1}) + J(4,\p_t v_{n-1}) + J(4,\p_{yy} v_{n-1}) + J(4,v_n) + J(2, \p_y v_n) + J(0,\p_t v_n) \\
+ J(0,\p_{yy} v_n) \leqslant C_1 I(3n+1,v_1) + C_2 I(3n-2,v_2) + C_3 J(3n-5, v_3) + \ldots + C_{n-1} I(8,v_{n-1}). \numberthis \label{eq_CE_nxn_B4}
\end{multline*}
Next, we estimate $I(8,v_{n-1})$ analogously, by $I(\cdot,v_{n-2})$. In general, we follow the same algorithm recursively, until we can estimate $I(\cdot,v_2)$. We provide some more details for this step. Let us define by $\rho_k$ the coefficient of the $I(\cdot,v_k)$ term used in the RHS of the reduced Carleman estimate while estimating $I(\cdot,v_{k+1})$ by $I(\cdot,v_{k})$. Then $\rho_k$ satisfies the following recursion relation $ \rho_{k-1} := 2 \rho_k - d_k + 4$, for $k=n-1,\ldots,2$. In the last step, we estimate $I(\rho_2,v_2)$ by $I(\rho_1,v_1)$. In conclusion, \eqref{eq_CE_nxn_B4} implies that
\begin{multline*}
J(3n+1,v_1) + J(3n-1,\p_y v_1) + J(3n-2,v_2) + J(3n-4, \p_y v_2) + J(3n-6,\p_t v_2) 
+ J(3n-6,\p_{yy} v_2) \\ + \ldots + J(7,v_{n-1}) + J(5, \p_y v_{n-1}) + J(4,\p_t v_{n-1}) + J(4,\p_{yy} v_{n-1}) + J(4,v_n) + J(2, \p_y v_n) + J(0,\p_t v_n) \\
+ J(0,\p_{yy} v_n) \leqslant C_1 I(\rho_1,v_1). 
\end{multline*}
Then, we can drop some terms from the LHS of the above estimate to get
\[J(3n+1,v_1) + J(3n-2,v_2) + \ldots + J(7,v_{n-1}) + J(4,v_n) \leqslant C_1 I(\rho_1,v_1).\]
Writing out the integrals using \eqref{eq_def_JI}, we get
\begin{equation} \label{eq_CE_nxn_B6}
\int_0^T \int_0^L (s \theta)^{3n+1} \left| e^{s \theta \psi} v_1 \right|^2 \rd y \rd t + \ldots + \int_0^T \int_0^L (s \theta)^4 \left| e^{s \theta \psi} v_n \right|^2 \rd y \rd t \leqslant C \int_0^T \int_\omega (s \theta)^{\rho_1} \left| e^{s \theta \psi} v_1 \right|^2 \rd y \rd t.
\end{equation}
\noindent \underline{\textbf{Step D}}:  Now we will estimate the Carleman weights present in the integrals and use energy estimates to finally prove the observability.
For $s_2 > s_1$ and $s = s_2 + C \| b \|_{L^\infty}^{2/3}$, \eqref{eq_CE_nxn_B6} gives us the following
\begin{multline*}
\int_0^T \int_0^L (s_2 \theta)^{3n+1}  e^{ 2 (s_2 + C \| b \|_{L^\infty}^{\frac{2}{3}}) \theta \psi} v_1^2 \rd y \rd t + \ldots + \int_0^T \int_0^L (s_2 \theta)^4 e^{2 (s_2 + C \| b \|_{L^\infty}^{\frac{2}{3}}) \theta \psi} v_n^2 \rd y \rd t \\
\leqslant C \int_0^T \int_\omega (s_2 + C \| b \|_{L^\infty}^{\frac{2}{3}})^{\rho_1} \theta^{\rho_1} e^{ 2 (s_2 + C \| b \|_{L^\infty}^{\frac{2}{3}}) \theta \psi} v_1^2 \rd y \rd t.
\end{multline*}
Now, we cannot obtain lower bound for $\theta$ near $t=0$ or $t=T$. Hence, we restrict the integral region over time to $[2T/3,3T/4]$ on the LHS, to get
\begin{multline*}
\int_{\frac{2T}{3}}^{\frac{3T}{4}} \int_0^L (s_2 \theta)^{3n+1}  e^{ 2 (s_2 + C \| b \|_{L^\infty}^{\frac{2}{3}}) \theta \psi} v_1^2 \rd y \rd t + \ldots + \int_{\frac{T}{4}}^{\frac{3T}{4}} \int_0^L (s_2 \theta)^4 e^{2 (s_2 + C \| b \|_{L^\infty}^{\frac{2}{3}}) \theta \psi} v_n^2 \rd y \rd t \\
\leqslant C \int_0^T \int_\omega (s_2 + C \| b \|_{L^\infty}^{\frac{2}{3}})^{\rho_1} \theta^{\rho_1} e^{ 2 (s_2 + C \| b \|_{L^\infty}^{\frac{2}{3}}) \theta \psi} v_1^2 \rd y \rd t.
\end{multline*}
We set the constant $C_3(T) := - \ds \inf_{y \in (0,L)} \psi(y) \cdot \ds \sup_{t \in [2T/3,3T/4]} \theta (t)$. Also note that the coefficient of the $v_1^2$ term on the RHS can be bounded from above, uniformly. Then, using $\tilde{C}_3$ and the above estimate, we get
\[C e^{-C_3 (s_2 + C \| b \|_{L^\infty}^{\frac{2}{3}})} \int_{\frac{2T}{3}}^{\frac{3T}{4}} \int_0^L (v_1^2 + \ldots + v_n^2) \ \rd y \rd t \leqslant \int_0^T \int_\omega v_1^2 \rd y \rd t.\]
Using the energy result from \eqref{eq_energy} and the above estimate, we get
\[C e^{-C_3 (s_2 + C \| b \|_{L^\infty}^{\frac{2}{3}})} e^{-C'\frac{2T}{3}} \left(\left\| v_1 \left(\frac{2T}{3} \right) \right\|_{L^2(0,L)}^2 + \ldots + \left\| v_n \left(\frac{2T}{3} \right) \right\|_{L^2(0,L)}^2 \right) \frac{T}{12} \leqslant \int_0^T \int_\omega v_1^2 \rd y \rd t,\]
which concludes the proof.
\end{proof}

Now, we will find the observability estimate for $w$, the solution of the adjoint system \eqref{adjointsystem_w}.
\begin{prop}
Let $w$ be the solution of \eqref{adjointsystem_w}. Then there exists constants $C, C', C_3>0$, such that the following is satisfied
\begin{equation} \label{eq_obs_est_w_n}
\left\| w_1 \left(\frac{2T}{3} \right) \right\|_{L^2(0,L)}^2 + \ldots + \left\| w_n \left(\frac{2T}{3} \right) \right\|_{L^2(0,L)}^2  \leqslant e^{C_3 \left(s_2 + C \| b \|_{L^\infty}^{\frac{2}{3}} + \frac{(M_1 - M_2)}{\varepsilon}\right)} e^{\frac{2C'T}{3}} \frac{12}{T} \int_0^T \int_\omega w_1^2 \rd y \rd t.
\end{equation}
\end{prop}

\begin{proof}
First, we will reverse the change of coordinates given by \eqref{eq_cov_wtov} and find the corresponding estimate of type \eqref{eq_obs_est_v_n} for $\tilde{w}$. Note that due to \eqref{eq_cov_wtov}, we have
\[ \tilde{w}(t,y) = e^{\ds \int_{0}^y \frac{g(s)}{2}ds } v(t,y).\]
Let $M_1$ and $M_2$ be constants such that 
\begin{align*}
\frac{M_2}{\varepsilon} \leqslant \int_{0}^y \frac{g(s)}{2}ds \leqslant \frac{M_1}{\varepsilon}.
\end{align*}
Then
\begin{equation} \label{eq_vtow_rhs_n}
\int_0^T \int_\omega v_1^2  \rd y \rd t \leqslant C e^{\frac{C M_1}{\varepsilon}} \int_0^T \int_\omega \tilde{w}_1^2 \rd y \rd t,
\end{equation}
and for the other side, we have
\[ e^{\frac{C M_2}{\varepsilon}} \left( \left\| \tilde{w}_1 \left(\frac{2T}{3} \right) \right\|_{L^2(0,L)}^2 + \ldots + \left\| \tilde{w}_n \left(\frac{2T}{3} \right) \right\|_{L^2(0,L)}^2 \right) \leqslant \left\| v_1 \left(\frac{2T}{3} \right) \right\|_{L^2(0,L)}^2 + \ldots + \left\| v_n \left(\frac{2T}{3} \right) \right\|_{L^2(0,L)}^2. \]
Combining \eqref{eq_obs_est_v_n}, \eqref{eq_vtow_rhs_n}, and the above estimate, we get
\begin{align*}
\left\| \tilde{w}_1 \left(\frac{2T}{3} \right) \right\|_{L^2(0,L)}^2 + \ldots + \left\| \tilde{w}_n \left(\frac{2T}{3} \right) \right\|_{L^2(0,L)}^2 & \leqslant e^{C_3 (s_2 + C \| b \|_{L^\infty}^{\frac{2}{3}})} e^{\frac{2C'T}{3}} e^{\frac{(M_1 - M_2)}{\varepsilon}} \frac{12}{T} \int_0^T \int_\omega \tilde{w}_1^2 \rd y \rd t\\
& \leqslant e^{C_3 \left(s_2 + C \| b \|_{L^\infty}^{\frac{2}{3}} + \frac{(M_1 - M_2)}{\varepsilon}\right)} e^{\frac{2C'T}{3}} \frac{12}{T} \int_0^T \int_\omega \tilde{w}_1^2 \rd y \rd t.
\end{align*}
Using \eqref{eq_w_tildew} to transform $\tilde{w}$ back to $w$, completes the proof of the proposition.
\end{proof}

\subsection{Uniform $L^2$-bound on the controls:}
Now we will combine the solutions obtained in \cref{ssec_stepi,ssec_stepii,ssec_stepiii} and conclude the uniform null controllability result. 

\begin{proof}[Proof of \cref{thm_control_1}]

For $\varepsilon\in(0,1),$  from the control strategy for the interval $[0,T/3]$ described in \cref{ssec_stepi}, we have the existence of $f^\varepsilon_1\in L^2((0,T/3)\times\omega)$, such that 
\begin{align}\label{stage1}
\begin{dcases}
\p_t u^\varepsilon- \mc{L}^\varepsilon u^\varepsilon +\mc{C}u^\varepsilon = e_1 1_{\omega}f^\varepsilon, & (t,x)\in \left(0, {T}/{3} \right)\times (0,1), \\ 
u^\varepsilon (t,0)=u^\varepsilon (t,1)=0, & t \in (0,T/3),\\
u(0,x)=u^0,\ \Pi_{H_{[D/\varepsilon]}} u(T/3, x)=0, & x \in (0,1).
\end{dcases}
\end{align}
Let us denote the solution of system \eqref{stage1} by $u^\varepsilon_{a}$.   
In the time interval $({T}/{3}, {2T}/{3})$, we allow the system to dissipate freely, that is 
\begin{align}\label{stage2} 
\begin{dcases}
\p_t u^\varepsilon- \mc{L}^\varepsilon u^\varepsilon +\mc{C}u^\varepsilon=0, & (t,x) \in (T/3, 2T/3 ) \times (0,1), \\ 
u^\varepsilon(t,0)=u^\varepsilon(t,1)=0, & t \in (T/3, 2T/3 ),\\
u^\varepsilon (T/3, x)=u_{a}^\varepsilon(T/3,x), & x \in (0,1).
\end{dcases}
\end{align}
Let us denote the solution of \eqref{stage2} by $u^\varepsilon_{b}$. Note that the data given at ${T}/{3}$ is in $H_{[D/\varepsilon]}^\perp$. From the analysis done in \cref{ssec_stepii} and \Cref{tby3bound}, , we have 
\begin{align*} 
\begin{split}
\|u^\varepsilon_{b}(t,\cdot)\|_{(L^2(0,1))^n} & \leqslant e^{-\mu_{[D/\varepsilon]}(t-{T}/{3})}\|u^\varepsilon_{a}(T/3)\|_{(L^2(0,1))^n} \\ & \leqslant C(T) e^{-\dfrac{C}{\varepsilon^2}(t-{T}/{3})}(\|u_0\|_{(L^2(0,1))^n} +\|f_1^\varepsilon \|_{L^2((0,T/3)\times \omega)}\|) \\ & \leqslant  C(T) e^{-\dfrac{C}{\varepsilon^2}(t-{T}/{3})} \|u_0\|_{(L^2(0,1))^n},
\end{split}
\end{align*}
where in the second inequality we have used the fact $\mu^\varepsilon_{[D/\varepsilon]}\sim \frac{1}{\varepsilon^2}$. In particular at ${2T}/{3}$, we have 
\begin{align} \label{decay}
\|(u^\varepsilon_{b}(2T/3,\cdot)\|_{(L^2(0,1))^n}\leqslant  e^{-\dfrac{CT}{\varepsilon^2}}\|u_0\|_{(L^2(0,1))^n}.
\end{align}
Now, we apply control on the time interval $(2T/3,T)$ to steer the trajectory  to $0$ at time T. 
Consider the controllable system starting from $2T/3$ as follows
\begin{align}\label{stage3}
\begin{dcases}
\p_t u^\varepsilon- \mc{L}^\varepsilon u^\varepsilon +\mc{C} u^\varepsilon = e_1 1_{\omega}f_2^\varepsilon, & (t,x) \in (2T/3,T)\times (0,1), \\ 
u^\varepsilon (t,0)=u^\varepsilon (t,1)=0, & t \in (2T/3,T),\\
u (2T/3,x) = u^\varepsilon_b (2T/3,x), & x \in (0,1).
\end{dcases}
\end{align}
From the analysis in \cref{ssec_stepiii}, in particular the observability  estimate \eqref{eq_obs_est_w_n}, there exists a control ${f}^\varepsilon_2$ for system \eqref{stage3}, satisfying
\begin{align*}\|{f}_2^\varepsilon\|_{L^2((2T/3,T) \times \omega)}\leqslant C_1 e^{\left(\dfrac{ C_2}{\varepsilon^{4/3}}+\dfrac{C_3}{\varepsilon^{2/3}} + \dfrac{C_4}{\varepsilon}\right)}\|u^\varepsilon_b (2T/3 , \cdot )\|_{(L^2(0,1))^n}.\end{align*}
Using the decay estimate \eqref{decay}, in the right hand side of the above estimate, we get
\[ \|{f}_2^\varepsilon\|_{L^2((2T/3,T) \times \omega)}\leqslant C(T) e^{\left(\dfrac{ C_2}{\varepsilon^{4/3}}+\dfrac{C_3}{\varepsilon^{2/3}} + \dfrac{C_4}{\varepsilon}-\dfrac{C}{\varepsilon^2}\right)} \|u_0\|_{(L^2(0,1))^n}.\] 
Hence, for each $\varepsilon>0$, we define the function $f^\varepsilon : (0,T) \times \omega \to \mathbb{R}$ as follows 
\begin{align*}
f^\varepsilon = 
\begin{cases}
f_1^\varepsilon, & \text{ in } (0,T/3)\times (0,1),\\
0, & \text{ in }(T/3,2T/3)\times (0,1),\\
f_2^\varepsilon, & \text{ in } (2T/3,T)\times (0,1),
\end{cases}
\end{align*}
and use this $f^\varepsilon$ as the control for the whole system.
From the definition of ${f}^\varepsilon,$ we conclude that $\|{f}^\varepsilon\|_{L^2((0,T) \times \omega)}\leqslant C(T),$ for some constant $C(T)$ that is independent of $\varepsilon$. This completes the proof.
\end{proof}

As a consequence of the uniform bound on the control, we have the following uniform observability estimate for the adjoint system, by duality.
\begin{theorem}
For $\varepsilon\in (0,1),$ there exists a constant $C(T)$, independent of $\varepsilon$, such that
\begin{align}\label{full-obs-est}
\|w^\varepsilon(0,\cdot)\|_{(L^2(0,1))^n}\leqslant C(T) \int_{0}^T\int_{\omega}|w_1^\varepsilon|^2 \rd x \rd t,
\end{align}
for every solution $w^\varepsilon=(w_1^\varepsilon,w_2^\varepsilon,\cdots, w_n^\varepsilon)$ to the adjoint system \eqref{adjointsystem_w}.
\end{theorem}

\begin{rem}
Due to \cref{rem_equiv}, the validity of \cref{thm_control_1} implies that \cref{thm_control_m} is true.
\end{rem}

\section{Homogenization: \cref{can_one_homo}}

In this section, we will prove the homogenization result   \Cref{can_one_homo} for the canonical system with one control. The is done by proving the following steps
\begin{enumerate}
\item By using the duality principle, we find a new minimal $L^2$-norm control at every $\varepsilon$-stage, and show that they are uniformly bounded.
\item Up to a subsequence, the weak limit of the new sequence of  null controls  is a null control for the homogenized system.
\item The subsequential weak limit of the null control sequence is the same as the control that can be found using the duality principle for the homogenized system. This shows the convergence of the full sequence.
\item Finally, we show the norm convergence of the null control sequence, which proves the strong convergence.
\end{enumerate}
Let us recall the adjoint system
\begin{align} \label{adjointsystem_w_H}
\begin{dcases}
- \p_t w^\varepsilon- \mc{L}^\varepsilon w^\varepsilon+\mc{C}^* w^\varepsilon=0, & (t,x)\in (0, T)\times (0,1), \\ 
w^\varepsilon(t,0)=w^\varepsilon
(t,1)=0, & t \in (0,T),\\
w^\varepsilon(T,x)=w^\varepsilon_0(x), & x \in (0,1).
\end{dcases}
\end{align}
For $\varepsilon\in (0,1),$ we define the following Hilbert space 
\[ H_\varepsilon:= \left\{ {w}^\varepsilon_0:  \text{the corresponding solution of }  \eqref{adjointsystem_w_H} \text{ satisfies } \int_0^T\int_\omega |{w_1}^\varepsilon (t,x)|^2 \rd x \rd t < \infty \right\}, \]
with the norm in $H_\varepsilon$ given by
\[\|w_0^\varepsilon\|^2_{H_\varepsilon}:= \int_0^T\int_\omega |{w_1}^\varepsilon (t,x)|^2 \rd x \rd t,\]
where $w^\varepsilon$ satisfies  $\eqref{adjointsystem_w_H}$.

By duality principle, for $\varepsilon\in (0,1),$ we will obtain the minimal $L^2$-norm null control by minimising the following cost functional over the space $H_\varepsilon$ 
\[J_\varepsilon({w}_0^\varepsilon)=\frac{1}{2}\int_{0}^T\int_{\omega}|{w_1}^\varepsilon|^2 \rd x \rd t +\int_{0}^1{u}^0 {w}^\varepsilon(0,x) \rd x.\]

For $\varepsilon\in (0,1),$ the cost functional $J_\varepsilon$ is convex and continuous. The  observability estimate \eqref{full-obs-est} shows the coercivity of $J_\varepsilon$ over the space $H_\varepsilon$.  Then there exists a minimiser  $\bar{w}_0^\varepsilon\in H_\varepsilon$ and a null control to the system \eqref{canonical_system} given by $1_\omega f=1_\omega \bar{w}_1^\varepsilon$, where $\bar{w}_\varepsilon$ is the solution to the adjoint system \eqref{adjointsystem_w_H} with initial data as the minimiser $\bar{w}_0^\varepsilon$.  Let us prove the following lemma. 

\begin{lemma}
For $\varepsilon\in (0,1),$ there exists a constant $C(T)$, independent $\varepsilon$, such that
$$\int_0^T\int_\omega |\bar{w}_1^\varepsilon (t,x)|^2 \rd x \rd t \leqslant C(T).$$
\end{lemma}
\begin{proof}
We have previously  shown that, for $\varepsilon\in (0,1), $ there is a control $f_\varepsilon\in L^2((0,T)\times\omega )$ and $\|f^\varepsilon\|_{(L^2((0,T)\times \omega))} \leqslant C(T)$. Due to minimality of $L^2$-norm properties, we have the desired result.
\end{proof}

Now we we will show that upto a subsequence, the weak limit of $1_\omega \bar{w}_1^\varepsilon$   is a null control to the homogenized system \eqref{eq_homogen_one}.
\begin{lemma}\label{weak_conv_control}
Let $\varepsilon \in (0,1)$ and $1_\omega w_\varepsilon^1$ be one of the null control with  minimal $L^2$-norm. Then $\exists \bar{f}\in L^2((0,T) \times \omega ),$ such that upto a subsequence (still denoted by $\varepsilon$)
\begin{align}\label{ep_identity}
\bar{w}_1^\varepsilon\rightharpoonup \bar{f}, \quad \text{weakly in }L^2((0,T)\times \omega), 
\end{align} 
where $\bar{f}$ is a null control to the homogenized system \eqref{eq_homogen_one}.
\end{lemma}
\begin{proof}
By $L^2$-weak compactness, $\exists \bar{f}\in L^2((0,T)\times \omega)$, such that 
$$\bar{w}_1^\varepsilon\rightharpoonup \bar{f}, \quad \text{weakly in }L^2((0,T)\times \omega).$$
Since $1_\omega \bar{w}_1^\varepsilon$ is obtained through  minimising  the functional $J_\varepsilon$ over $H_\varepsilon$, we have have the following identity
\begin{align}\label{ep_char_iden}
\int_0^T\int_\omega\bar{w}_1^\varepsilon w_1^\varepsilon \rd x \rd t + \int_0^1 u_0(x) w^\varepsilon(0,x) \rd x = 0,
\end{align}
where $w^\varepsilon$ satisfies the adjoint system \eqref{adjointsystem_w_H}.
Let us fix $w_0^\varepsilon=w_0\in L^2(0,1)$. Then, by using a standard homogenization technique (see  \cite[Chapter 11]{PD99}) the solution $w^\varepsilon$ satisfies the following convergences
\begin{align}\label{ad_data_conv}
w_1^\varepsilon \to w_1, \quad \text{strongly in } L^2((0,T)\times \omega), \text{ and }
w^\varepsilon(\cdot,0) \to w(\cdot, 0), \quad \text{strongly in } L^2(0,1),
\end{align}
where $w=(w_1,w_2,\cdots,w_n)^{\text{tr}}$ is the solution to the homogenized adjoint system
\begin{align}\label{eq_ad_homoge_one}
\begin{dcases}
-\p_t {w}-\mc{L}^0 {w} +\mathcal{C}^* {w} = 0, & (t,x)\in (0, T)\times (0,1), \\ 
{w}(t,0)={w}(t,1)=0, & t \in (0,T),\\
{w} (0,x)={w}^0, & x \in (0,1).
\end{dcases}
\end{align} 
Hence using \eqref{ad_data_conv}, and taking $\varepsilon \to 0,$ identity \eqref{ep_identity} reduces to the following
\begin{align}\label{f_identity}
\int_0^T\int_\omega \bar{f}w_1\rd x \rd t +\int_0^1 u_0 w(0,x)\rd x =0,
\end{align}
for all solution $w$ of \eqref{adjointsystem_w_H} with $w_0\in L^2(0,1).$ Hence, $\bar{f}$ is a null control to the homogenized system \eqref{eq_homogen_one}.
\end{proof}

Now, we are ready to prove the main homogenization result. 
\begin{proof}[Proof of \cref{can_one_homo}:]
From \Cref{weak_conv_control}, upto a subsequence, the  weak limit of the $\varepsilon$-level null control is a null control for the homogenized system. Hence, it remains to prove that the weak limit is unique and
\begin{align*}
1_{\omega}\bar{w}_1^\varepsilon \to 1_\omega \bar{f}, \quad \text{ strongly in }  L^2((0,T)\times \omega).
\end{align*}
First, we identify $\bar{f.}$ For this purpose, let us define the following function space
$$H_0:= \left\{ {w}_0:  \text{the corresponding solution of } \eqref{eq_ad_homoge_one} \text{ satisfies } \int_0^T\int_\omega |{w}|^2 \rd x \rd t < \infty \right\}. $$
One of the ways to find a null control for the homogenized system \eqref{eq_homogen_one} is by minimising the functional
\[J_0({w}_0)=\frac{1}{2}\int_{0}^T\int_{\omega}|{w_1}|^2 \rd x \rd t + \int_{0}^1{u}^0 {w}(0,x) \rd x,\]
 over the space $H_0$. 
Note that, following the analysis in \Cref{ssec_stepiii}, the following observability estimate holds for the  system \eqref{eq_ad_homoge_one}
\begin{align}\label{homo_obs_est}
\|w(0,\cdot)\|^2_{(L^2(0,1))^n}\leqslant C(T)\int_0^T\int_\omega|w_1|^2\rd x \rd t,
\end{align} 
for every solution $w=(w_1,w_2,\cdots,w_n)^\text{tr}$. 

It is standard to see that the functional $J$ is continuous and  convex over the space $H_0.$ The coercivity follows from the observability estimates \eqref{homo_obs_est}. Hence, there exists a unique minimiser of $J_0$ in $H_0.$  Let $\hat{w}_0 \in H_0$ is the minimiser and we denote corresponding solution of \eqref{eq_ad_homoge_one} by $\hat{w}$. The null control is given by $1_\omega \hat{w}_1$ and it is characterised by the following identity
\begin{align}\label{homo_identity}
\int_0^T\int_\omega \hat{w}_1w_1\rd x \rd t+\int_0^1 u_0(x)w(0,x)\rd x =0,
\end{align}
for all solution $w$ of \eqref{adjointsystem_w_H}.  Comparing the identity \eqref{f_identity} and \eqref{homo_identity}, we have 
\begin{align*}
1_\omega\bar{f}=1_\omega \hat{w}_1.
\end{align*}
From the uniform bound on the control, we have $\|\bar{w}_0^\varepsilon\|_{H_\varepsilon}\leqslant C$. Thanks to \cite[Theorem 1.1]{teb12}, we also have $\|w_\varepsilon^0\|_{H_0}\leqslant C$.  By weak compactness, there exists a $\bar{w}_0\in H_0$ such that 
\[\bar{w}_0^\varepsilon\rightharpoonup \bar{w}_0, \quad \text{ weakly in } H_0.\]
From the observability inequality \eqref{full-obs-est}, we have, for any $0\leqslant\tau <T$
$$\|{\bar{w}}^\varepsilon(\tau,\cdot)\|_{(L^2(0,1))^n}\leqslant C(T-\tau)\int_{\tau}^{T}\int_{\omega}|{\bar{w}_1}^\varepsilon|^2 \rd x \rd t \leqslant C(T-\tau).$$
Due to regularising effect of the parabolic system, we obtain the following estimate
$$\|{\bar{w}}^\varepsilon(\tau,\cdot)\|_{(H^2(0,1)\cap H^1(0,1))^n} \leqslant C(T-\tau).$$
By using classical energy estimates for the parabolic system,  up to subsequences, we have
\begin{align}
&{\bar{w}}_\varepsilon \to \bar{w}, \quad \text{strongly in }  (L^2((0,\tau)\times (0,1)))^n,~~\text{ for  }0 \leqslant \tau <T, \notag \\
& \bar{w}^\varepsilon (0,\cdot)\to \bar{w}(0,\cdot), \quad \text{strongly in } L^2(0,1),
\label{initial_data_conv} 
\end{align}
where $\bar{w}=(\bar{w}_1,\bar{w}_2,\cdots,\bar{w}_n)$ satisfies the following homogenized system \eqref{eq_ad_homoge_one} with $\bar{w}(T,x)=\bar{w}_0(x).$ This shows that
$$1_\omega \bar{f}=1_\omega \bar{w}_1=\hat{w}_1.$$
Hence $\bar{w}_0=\hat{w}_0.$ 
Using the convergence \eqref{initial_data_conv}, identities \eqref{ep_char_iden} and \eqref{homo_identity}, we obtain
\begin{align*}
\limsup_{\varepsilon\to 0}\int_0^T\int_\omega |\bar{w}^\varepsilon|^2 \rd x \rd t=&\limsup_{\varepsilon \to 0} \left(-\int_0^1u_0(x)\bar{w}^\varepsilon(0,x) \rd x \right) \\
&=-\int_0^1u_0(x)\bar{w}(0,x)\rd x =\int_0^T\int_\omega |\bar{w}_1|^2 \rd x \rd t \leqslant \liminf_{\varepsilon \to 0} \int_0^T\int_\omega |\bar{w}_1^\varepsilon| \rd x \rd t.
\end{align*}
This implies that
\begin{align*}
\int_{0}^T\int_\omega|\bar{w}_1^\varepsilon|^2 \rd x \rd t \to \int_0^T\int_\omega |\bar{w}_1|^2 \rd x \rd t.
\end{align*}
Hence, weak convergence together with norm convergence implies the strong convergence of the null control sequence.  The uniqueness of the limit follows from the uniqueness of the minimiser $\bar{w}_0 \in H_0$. The uniqueness of the limit proves the convergence of the full  sequence  $\{\bar{f}^\varepsilon\}$.
\end{proof}

\begin{rem}
Combining \cref{can_one_homo} with \cref{rem_equiv} implies that \cref{thm_homogen_m} is true.
\end{rem}

\section{Concluding remarks}
The operator $\mc{L}^\varepsilon$ given by \eqref{eq_L-op} involves the same diffusion coefficient $a^\varepsilon$ in all the components. An interesting problem is to consider the case when one has different coefficents, say $\p_{x_i} (a^\varepsilon_i(x) \p_{x_i}) I_{n\times n}$. We mention that this leads to additional difficulties and it seems that the  present approach does not apply to this new problem. First, it is difficult to find invariant spaces, for the case of different diffusion coefficients, in the spectral analysis of the coupled system required to control the low frequencies. Second, note that, the change of variable used in \eqref{eq_change_of_var} depends on the coefficient $a$. Hence, if we have different coefficients then this step breaks down. A possible solution is to perhaps find a change of variable \emph{matrix} that transforms the system in a similar manner as \eqref{eq_change_of_var}. Another method that can work is to change the order of operations in the analysis; for example, corresponding to \cref{ssec_stepiii}, one can first do a change of variable dependent on $a^\varepsilon_i$, for each component separately. This has to be followed by proving an observability via Carleman estimate for each component, then reversing the change of variable using $a^\varepsilon_i$, and finally estimating the resulting terms accordingly.

\appendix

\section{Behaviour of eigenvalues and eigenfunctions} \label{app_eigen}
For $\varepsilon\in (0,1)$,	we want to analyze the behaviour of eigenvalues and eigenfunctions of the  operator: $\ds -\left(a\left(\frac{x}{\varepsilon}\right)\phi_x\right)_x$. The eigenvalue problem for each $\varepsilon>0$ is the following
\begin{align}\label{div_evp}
\begin{dcases}
-\left(a\left(\frac{x}{\varepsilon}\right)\phi_x\right)_x = \lambda \phi,\\
\quad \phi(0) = \phi(1) = 0.
\end{dcases}
\end{align}
In order to study the above problem, first we transform  it to a well studied eigenvalue problem using the following change of variable
\[z :=\frac{\ds \int_0^{ \frac{x}{\varepsilon}} \frac{1}{a(s)} \rd s}{ \ds \int_0^{\frac{1}{\varepsilon} }\frac{1}{a(s)} \rd s}=\delta(\varepsilon)h\left(\frac{x}{\varepsilon} \right), \quad \text{ where } \delta(\varepsilon) := \left(\int_0^{ \frac{1}{\varepsilon}} \frac{1}{a(s)}\rd s \right)^{-1}.\]
Then system \eqref{div_evp} reduces to the following system
\begin{align}\label{changed_sys}
\begin{cases}
& \ds- \phi_{zz}=\left(\frac{\varepsilon}{(\delta(\varepsilon))}\right)^2 a\left(h^{-1}\left(\frac{z}{\delta(\varepsilon)}\right)\right) \lambda \phi, \qquad z \in (0,1),\\
& \phi(0)=\phi(1)=0.
\end{cases}
\end{align}
For simplicity, we can consider the sequence $\ds\varepsilon=\frac{1}{n},~~n\in \mathbb{N} .$ Then 
\begin{align*}
\frac{\varepsilon}{\delta(\varepsilon)}=\int_{0}^{1}\frac{1}{a(s)}ds =l.
\end{align*}
Let us define the function $\tilde{a}$ as 
\begin{align}\label{atilde}
\ds \tilde{a}\left(\frac{z}{\varepsilon}\right) := \left(\frac{\varepsilon}{\delta(\varepsilon)}\right)^2 a\left(h^{-1}\left(\frac{z}{\delta(\varepsilon)}\right)\right).\end{align} 
The system \eqref{changed_sys}, reduces to
\begin{align}\label{final_sys}
\begin{cases}
& \ds- \phi_{zz}=\tilde{a}\left(\frac{z}{\varepsilon}\right) \lambda \phi, \qquad z \in (0,1),\\
& \phi(0)=\phi(1)=0.
\end{cases}
\end{align}
Let $(\lambda_k^\varepsilon,{\tilde{\phi}}^{\varepsilon}_k)$ be the $k$-th eigenpair to \eqref{final_sys}. Let us write $\varphi^\varepsilon_k(x)={\tilde{\phi}}^{\varepsilon}_k (\delta(\varepsilon)h\left(\frac{x}{\varepsilon}\right))$. Then $(\lambda_k^\varepsilon, \varphi_k^\varepsilon)$ is the $k$-th eigenpair corresponding to system \eqref{div_evp}.

\subsection{Uniform positivity of $\|\phi_k^\varepsilon\|_{L^2(\omega)}$} We provide a brief proof here, the detailed explanation can be found in \cite{castro_zuazua}.
Let $k\leqslant [D/\varepsilon]$. Using the WKB expansion as in \cite{castro_zuazua}, we can write the $k$-th eigenfunction for system \eqref{final_sys}, corresponding to $\lambda_k$ as follows
\begin{align*}
\phi_k^\varepsilon(x)=\gamma_k \text{Im}(\exp\left(\sum_{n=0}^\infty \ds \varepsilon^{n+1}\lambda_k^{\frac{n+1}{2}} S^n \left(\frac{x}{\varepsilon}\right)\right),
\end{align*}
where $\gamma_k$ is the normalising factor and the functions $S^n$'s satisfy the following system
\begin{align*}
\begin{dcases}
S^0_{xx}=0,~~~ &\text{ in } \mathbb{R},  \\
 S^1_{xx}+(S^0_x)^2+\tilde{a}(x)=0, ~~&\text{ in } \mathbb{ R},\\
S^n_{xx}+\sum_{i+j=n-1}S^{i}_xS^j_x=0~~~&\text{ in } \mathbb{ R},
\end{dcases}
\end{align*}
with $S^n_x$ being $1$-periodic.
It has been shown in \cite{castro_zuazua} that
\begin{align*}
\phi_{k}^\varepsilon(x)=\gamma_k\exp(\text{Re}(S^\varepsilon(x))) \sin(\text{Im}(S^\varepsilon(x))), \quad \text{ where } S^\varepsilon(x)=\sum_{n=0}^\infty \ds \varepsilon^{n+1}\lambda_k^{\frac{n+1}{2}} S^n \left(\frac{x}{\varepsilon}\right).
\end{align*}
Furthermore, in \cite{castro_zuazua}, the following estimates are also established
\begin{align*}
& |\text{Re}(S^\varepsilon(x))|\leqslant Ck\varepsilon,\\
& \text{Im} (S^\varepsilon(x))=(1-x)k\pi +O(\varepsilon^3k^3),\\
&\text{Im} ((S^\varepsilon)'(x))=-k \pi +O(\varepsilon^2 k^3).
\end{align*} 
To have uniform lower bound of $\|\phi_k^\varepsilon\|_{L^2(\omega)}$, it is enough to show that   $\|\sin(\text { Im }(S^\varepsilon(x)))\|_{L^2(\omega)}$ has uniform lower bound. Consider
\begin{align*}
\int_{\omega}\sin^2( \text{Im} (S^\varepsilon(x)) \rd x & = \frac{|\omega|}{2}-\int_{\omega} \cos(2 \text{Im} (S^\varepsilon(x)) \rd x \\
&=\frac{|\omega|}{2}- \frac{\sin( \text{Im} (S^\varepsilon(x)))}{2(\text{Im} ((S^\varepsilon)'(x)))}\bigg|_{\partial \omega}=\frac{|\omega|}{2}- \frac{\sin( \text{Im} (S^\varepsilon(x)))}{2(-k \pi +O(\varepsilon^2 k^3))}\bigg|_{\partial \omega}\\  &\geqslant \frac{|\omega|}{2}-O\left(\frac{1}{k}\right)\geqslant \frac{|\omega|}{4}, ~~\text{ for some fixed } k>k_0.
\end{align*}
In the above calculation we have used the fact that $k\varepsilon <D$. Hence, $O(\varepsilon^2 k^3)\simeq O(k)$. 
For $\varepsilon>0$, small enough, there exists $\delta>0$, such that, $\inf\{ \|\phi_k^\varepsilon\|^2_{L^2(\omega)}:~0 \leqslant k\leqslant k_0\}>\delta.$ The existence of $\delta$ can be shown by observing that, for small enough $\varepsilon>0,$  and $k\leqslant k_0$, $\sin((1-x)k\pi+O(\varepsilon^3k^3))\simeq \sin((1-x)k\pi)$. Hence,  $\|\phi_k^\varepsilon\|_{L^2(\omega)}\geqslant \inf\{ \frac{\omega}{2}, \delta \}$ for $k\leqslant [D/\varepsilon]$.

\begin{prop}\label{spectral_ gap}
Assume that $a \in W^{2,\infty}(0,1)$ and satisfies \eqref{upper_lower_bound_1}. For $\varepsilon>0$, $(\lambda_k^\varepsilon, \varphi_k^\varepsilon)$ denotes  the $k$-th eigenpair corresponding to the system \eqref{div_evp}. Then, given any $\delta>0$, there exists a constant $C(\delta)$, independent of $\varepsilon$, such that  
\[\sqrt{\lambda^\varepsilon_{k+1}}-\sqrt{\lambda^\varepsilon_k}\geqslant \frac{\pi}{\sqrt{\bar{\tilde{a}}}}-\delta,~~~\forall~ k\varepsilon \leqslant  C(\delta).\]
where \begin{align}\label{avg-tilde-a}
\bar{\tilde{a}}=\int_0^1\tilde{a}(z) \rd z.
\end{align}
Furthermore, there exist $C$ and $\tilde{C}>0$ such that
\begin{align*}
\|\varphi_{k}^{\varepsilon}\|_{L^2(\omega)} \geqslant \tilde{C},~~\forall ~k\varepsilon \leqslant C,
\end{align*}
where $\omega$ is any nonempty open subset of $(0,1).$
\end{prop}

For the the operator $-\mathcal{L}^\varepsilon+C^*$ the above result can be used component-wise  with minor modification to obatin a proof of \cref{spectral_thm_couple}.

\section{$m$-control and $1$-control equivalence} \label{app_mcontrol}

In this section, we will show that the $m$-control system can be written as a \emph{combination} of $1$-control systems. The proof is analogous to the work presented in \cite{khodjaBDG}.

Suppose $B$ has the form $B= (b^1|b^2|\ldots |b^m)$, with $b^i \in \mb{R}$, for $1\leqslant i \leqslant m$.

\begin{lemma} \label{lm_basis}
Let (H1) be satsified, and let $\mf{X}$ denote the vector space generated by the columns of the matrix $[A|B]$. Then, there exist $r\in\{1,\ldots,k\}$, $\{l_j\}_{1\leqslant j \leqslant r} \subset \{1,\ldots,m\} $, and $\{ s_j \}_{1\leqslant j \leqslant r} \subset \{1,\ldots,n\}$ with $ \sum_{j=1}^r s_j=k$, such that
\[ \mf{B} := \bigcup_{j=1}^r \{ b^{l_j},Ab^{l_j},\ldots,A^{s_j-1}b^{l_j}\}\]
is a basis for $\mf{X}$.
For every $j$, there exists $a_{k,s_j}^i \in \mb{R}, 1\leqslant i \leqslant j, 1 \leqslant k \leqslant s_j $, such that
\begin{equation} \label{eq_mcontrol_decomp}
A^{s_j} b^{l_j} = \sum_{i=1}^j ( a^i_{1,s_j} b^{l_i} + a^i_{2,s_j} A b^{l_i} + \cdots + a^i_{s_i,s_j} A^{s_i-1} b^{l_i} ).
\end{equation}
\end{lemma}
We omit the proof of the above lemma, see \cite{khodjaBDG} for a constructive proof and a method for finding the basis $\mf{B}$.
Next, we define $P$ as the matrix given by
\[P := (b^{l_1}| A b^{l_1} | \ldots | A^{s_1-1} b^{l_1} | \ldots | b^{l_r} | A b^{l_r} | \ldots | A^{s_r-1}b^{l_r} ).\]
Now, due to \eqref{eq_mcontrol_decomp}, we have 
\[AP=PC, \quad P e_{S_i} = b^{l_i}, \ 1 \leqslant i \leqslant r,\]
where $S_i = 1 + \sum_{j=1}^{i-1} s_j$, for $ 1 \leqslant i \leqslant r$, and the matrix $\mc{C}$ is given by
\[\mc{C}= \begin{bmatrix}
C_{11} & C_{12} & \ldots & C_{1r} \\
0 & C_{22} & \ldots & C_{2r} \\
\vdots & \vdots & \ddots & \vdots \\
0 & 0 & \ldots & C_{rr}
\end{bmatrix},\]
with the matrices $C_{ii}$ and $C_{ij}$, for $1 \leqslant i \leqslant j \leqslant r$, are as follows
\[ C_{ii} = \begin{bmatrix}
 0 & 0 & 0 & \ldots & a^i_{1,s_i} \\
 1 & 0 & 0 & \ldots & a^i_{2,s_i} \\
 0 & 1 & 0 & \ldots & a^i_{3,s_i} \\
 \vdots & \vdots & \ddots & \ddots & \vdots \\
 0 & 0 & \ldots & 1 & a^i_{s_i,s_i}
\end{bmatrix}, \qquad 
C_{ij} = \begin{bmatrix}
 0 & 0 & 0 & \ldots & a^i_{1,s_j} \\
 0 & 0 & 0 & \ldots & a^i_{2,s_j} \\
 0 & 0 & 0 & \ldots & a^i_{3,s_j} \\
 \vdots & \vdots & \ddots & \ddots & \vdots \\
 0 & 0 & \ldots & 0 & a^i_{s_i,s_j}
\end{bmatrix}. \] 
Now let us consider the system 
\begin{align}\label{general_tildeB_M}
\begin{dcases}
\p_t \tilde{u}-\mc{L}^\varepsilon \tilde{u} +A \tilde{u}= \tilde{B} f 1_{\omega}, & (t,x)\in (0, T)\times (0,1), \\ 
\tilde{u}(t,0)=\tilde{u}(t,1)=0, & t \in (0,T) \\
\tilde{u}(0,x)=\tilde{u}^0, & x \in (0,1),
\end{dcases}
\end{align}
where the matrix $\tilde{B}$ is given by $\tilde{B} = (0|\ldots|0|b^{l_1}|0|\ldots|0|b^{l_2}|0|\ldots|0|b^{l_r}|0|\ldots) $, lying in the space $\mc{L}(\mb{R}^m;\mb{R}^n)$. If the above system, with control matrix $\tilde{B}$, is controllable, then system \eqref{general_B_M}, with control matrix $B$, is also controllable.

Now consider the transformation $u = P^{-1} \tilde{u}$, with $\tilde{f}=(\tilde{f}_1,\ldots,\tilde{f}_m)$ in $L^2((0,T)\times (0,1))^m$. Proving the null controllability of \eqref{general_tildeB_M} with $\tilde{f}$ as the control function, is equivalent to proving the null controllability of \eqref{canonical_system_n} with $D=(e_{S_1}|e_{S_2}|\ldots|e_{S_r})$ and the control function $f = (\tilde{f}_{l_1},\tilde{f}_{l_2},\ldots,\tilde{f}_{l_r})$. Note that, this equivalence allows us to write system \eqref{general_B_M} as a combination of $r$-subsystems. For each such subsystem the coupling is given by $C_{ii}$ and it is controlled by only one control vector given by $e_1^{s_i}=(1,0,\ldots,0)^\text{tr} \in \mb{R}^{s_i}$. The construction is such that each of these subsystems contains the pair $(C_{ii},e_1^{s_i})$, which satisfies the Kalman condition $\det [C_{ii}|e_1^{s_i}] \neq 0$. Furthermore, each of the coupling matrices $C_{ii}$'s, for $1\leqslant i \leqslant r$, are cascade type. 

Hence, to solve the control problem for system \eqref{general_B_M}, which has a general coupling  matrix with $m$-controls acting on it, it is enough to study a system with a cascade type coupling matrix with only one control acting on it, namely system \eqref{canonical_system}.


\begin{thebibliography}{10}

\bibitem{khodjaBDG}
{\sc F.~Ammar~Khodja, A.~Benabdallah, C.~Dupaix, and M.~Gonz\'{a}lez-Burgos},
  {\em A generalization of the {K}alman rank condition for time-dependent
  coupled linear parabolic systems}, Differ. Equ. Appl., 1 (2009),
  pp.~427--457.

\bibitem{MR2511553}
{\sc F.~Ammar-Khodja, A.~Benabdallah, C.~Dupaix, and M.~Gonz\'{a}lez-Burgos},
  {\em A {K}alman rank condition for the localized distributed controllability
  of a class of linear parbolic systems}, J. Evol. Equ., 9 (2009),
  pp.~267--291.

\bibitem{MR2846087}
{\sc F.~Ammar-Khodja, A.~Benabdallah, M.~Gonz\'{a}lez-Burgos, and
  L.~de~Teresa}, {\em Recent results on the controllability of linear coupled
  parabolic problems: a survey}, Math. Control Relat. Fields, 1 (2011),
  pp.~267--306.

\bibitem{boyer}
{\sc F.~Boyer}, {\em {Controllability of linear parabolic equations and
  systems}}.
\newblock Lecture, Feb. 2022.

\bibitem{castro_zuazua}
{\sc C.~Castro and E.~Zuazua}, {\em Low frequency asymptotic analysis of a
  string with rapidly oscillating density}, SIAM J. Appl. Math., 60 (2000),
  pp.~1205--1233.

\bibitem{PD99}
{\sc D.~Cioranescu and P.~Donato}, {\em An introduction to homogenization},
  Oxford university press, 1999.

\bibitem{coron_book}
{\sc J.-M. Coron}, {\em Control and Nonlinearity}, American Mathematical
  Society, USA, 2007.

\bibitem{coron09}
{\sc J.-M. Coron and S.~Guerrero}, {\em Null controllability of the
  n-dimensional stokes system with n- 1 scalar controls}, Journal of
  Differential Equations, 246 (2009), pp.~2908--2921.

\bibitem{coron10}
{\sc J.-M. Coron, S.~Guerrero, and L.~Rosier}, {\em Null controllability of a
  parabolic system with a cubic coupling term}, SIAM journal on control and
  optimization, 48 (2010), pp.~5629--5653.

\bibitem{dol_rus}
{\sc S.~Dolecki and D.~L. Russell}, {\em A general theory of observation and
  control}, SIAM J. Control Optim., 15 (1977), pp.~185--220.

\bibitem{jose15}
{\sc P.~Donato and E.~C. Jose}, {\em Asymptotic behavior of the approximate
  controls for parabolic equations with interfacial contact resistance}, ESAIM:
  Control, Optimisation and Calculus of Variations, 21 (2015), pp.~138--164.

\bibitem{jose21}
{\sc P.~Donato, E.~C. Jose, and D.~Onofrei}, {\em On the approximate
  controllability of parabolic problems with non-smooth coefficients},
  Asymptotic Analysis, 122 (2021), pp.~395--402.

\bibitem{don-nabil01}
{\sc P.~Donato and A.~Nabil}, {\em Approximate controllability of linear
  parabolic equations in perforated domains}, ESAIM: Control, Optimisation and
  Calculus of Variations, 6 (2001), pp.~21--38.

\bibitem{Fella19}
{\sc L.~Faella, S.~Monsurr{\`o}, and C.~Perugia}, {\em Exact controllability
  for evolutionary imperfect transmission problems}, Journal de
  Math{\'e}matiques Pures et Appliqu{\'e}es, 122 (2019), pp.~235--271.

\bibitem{fur_iman}
{\sc A.~V. Fursikov and O.~Y. Imanuvilov}, {\em Controllability of evolution
  equations}, vol.~34 of Lecture Notes Series, Seoul National University,
  Research Institute of Mathematics, Global Analysis Research Center, Seoul,
  1996.

\bibitem{MR2598471}
{\sc M.~Gonz\'{a}lez-Burgos and L.~de~Teresa}, {\em Controllability results for
  cascade systems of {$m$} coupled parabolic {PDE}s by one control force},
  Port. Math., 67 (2010), pp.~91--113.

\bibitem{MR2141924}
{\sc H.~Leiva}, {\em Controllability of a system of parabolic equations with
  non-diagonal diffusion matrix}, IMA J. Math. Control Inform., 22 (2005),
  pp.~187--199.

\bibitem{lopez_zuazua}
{\sc A.~L\'{o}pez and E.~Zuazua}, {\em Uniform null-controllability for the
  one-dimensional heat equation with rapidly oscillating periodic density},
  Ann. Inst. H. Poincar\'{e} C Anal. Non Lin\'{e}aire, 19 (2002), pp.~543--580.

\bibitem{ped06}
{\sc P.~Pedregal and F.~Periago}, {\em Some remarks on homogenization and exact
  boundary controllability for the one-dimensional wave equation}, Quarterly of
  applied mathematics, 64 (2006), pp.~529--546.

\bibitem{MR4687430}
{\sc T.~Takahashi, L.~de~Teresa, and Y.~Wu-Zhang}, {\em A {K}alman condition
  for the controllability of a coupled system of {S}tokes equations}, J. Evol.
  Equ., 24 (2024), pp.~Paper No. 4, 25.

\bibitem{teb12}
{\sc L.~Tebou}, {\em Uniform null controllability of a parabolic equation with
  rapidly oscillating periodic coefficients}, Asymptotic Analysis, 80 (2012),
  pp.~149--170.

\bibitem{tenenbaum_tucsnak}
{\sc G.~Tenenbaum and M.~Tucsnak}, {\em New blow-up rates for fast controls of
  {S}chr\"{o}dinger and heat equations}, J. Differential Equations, 243 (2007),
  pp.~70--100.

\bibitem{tucsnak_weiss}
{\sc M.~Tucsnak and G.~Weiss}, {\em Observation and control for operator
  semigroups}, Birkh\"{a}user Advanced Texts: Basler Lehrb\"{u}cher.
  [Birkh\"{a}user Advanced Texts: Basel Textbooks], Birkh\"{a}user Verlag,
  Basel, 2009.

\bibitem{zuazua94}
{\sc E.~Zuazua}, {\em Approximate controllability for linear parabolic
  equations with rapidly oscillating coefficients}, Control and Cybernetics, 23
  (1994).

\end{thebibliography}
\end{document}